\documentclass[12pt]{amsart}
 \usepackage{graphicx}
\usepackage{xstring}
\usepackage{forloop}
\usepackage{bbm, amsmath, amsfonts, amscd, latexsym, amsthm, amssymb, graphicx, mathrsfs, comment}
\usepackage[colorlinks]{hyperref}

\hypersetup{
  citecolor = blue,           
  linkcolor = red,          
  urlcolor  = blue
}
\usepackage{lmodern}            
\usepackage{amsmath, amssymb, amsthm}
\usepackage{mathtools}
\usepackage{enumitem}
\usepackage{hyperref}
\usepackage{graphicx}
\usepackage{tikz-cd}
\usepackage{algorithm}
\usepackage{algorithmic}
\usepackage{mathrsfs}
\usepackage[margin=2in]{geometry}

\usepackage{mathtools}
\usepackage{caption}
\usepackage{MnSymbol}

\usepackage{tikz-cd}
\tikzset{
  symbol/.style={
    draw=none,
    every to/.append style={
      edge node={node [sloped, allow upside down, auto=false]{$#1$}}}
  }
}

\makeatletter
\newif\if@check@engine  \@check@enginetrue 
\makeatother
\usepackage[usenames,dvipsnames]{pstricks}

\newcommand{\nocontentsline}[3]{}
\newcommand{\tocless}[2]{\bgroup\let\addcontentsline=\nocontentsline#1{#2}\egroup}

\usepackage{booktabs}

\newtheorem{theor}{\hspace{1cm}{\sc Theorem}}[section]
\newtheorem{prop}[theor]{\hspace{1cm}{\sc Proposition}}
\newtheorem{coro}[theor]{\hspace{1cm}{\sc Corollary}}
\newtheorem{lemma}[theor]{\hspace{1cm}{\sc Lemma}}

\newtheorem*{prop*}{\hspace{1cm}{\sc Proposition}}
\theoremstyle{definition}

\newtheorem{defin}[theor]{\hspace{1cm}{\sc Definition}}
\newtheorem*{defin*}{\hspace{1cm}{\sc Definition}}
\newtheorem{exa}[theor]{\hspace{1cm}{\sc Example}}

\newtheorem{rem}[theor]{\hspace{1cm}{\sc Remark}}

\newtheorem{quest}[theor]{\hspace{1cm}{\sc Question}}

\newtheorem{convention}[theor]{\hspace{1cm}{\sc Conventions}}

\newcounter{idx}

\newcommand{\rotraise}[1]{
  \StrLen{#1}[\slen]
  \forloop[-1]{idx}{\slen}{\value{idx}>0}{
    \StrChar{#1}{\value{idx}}[\crtLetter]
    \IfSubStr{tlQWERTZUIOPLKJHGFDSAYXCVBNM}{\crtLetter}
      {\raisebox{\depth}{\rotatebox{180}{\crtLetter}}}
      {\raisebox{1ex}{\rotatebox{180}{\crtLetter}}}}
}

\renewcommand{\emph}[1]{{\it {\color{NavyBlue} #1}}}

\def\R{\mathbb R}

\usepackage{xr-hyper}
\makeatletter

\newcommand*{\addFileDependency}[1]{%
\typeout{(#1)}%

\@addtofilelist{#1}
\IfFileExists{#1}{}{\typeout{No file #1.}}
}\makeatother

\usepackage[dvipsnames]{xcolor}
\usepackage[backend=bibtex,style=numeric,sorting=none]{biblatex}
\begin{document}

\title[Regular tropical fans]{Tropical fans supporting a reduced 0-dimensional complete intersection}

\author{Linxuan Li}

\begin{abstract}
An affine tropical fan is called regular if it supports a reduced 0-dimensional complete intersection. For some cases the classification of regular fans is already complete. It was proved by Fink that tropical varieties of degree 1 are exactly Bergman fans, and later Esterov and Gusev classified all lattice polytopes whose mixed
volume equals $1$. We introduce the notion of a \textit{gallery} for tropical fans and use it
to classify all one-dimensional regular fans, thereby obtaining a ‘minimal model programme' for such fans.
In dimension~$2$ we prove a finiteness theorem: every regular fan that satisfies the given upper bound condition is precisely the support of a finite covering by two-dimensional galleries, and only finitely many such fans exist.

\end{abstract}

\maketitle
\tableofcontents

\section{Introduction}

Tropical fans are important devices in combinatorial geometry: for example, every matroid $M$ can be encoded by its Bergman fan $[M]$, and every convex lattice polytope $P$ by its dual fan $[P]$. The set $\mathcal{T}$ of all tropical fans in $\R^n$ has a natural ring structure (see e.g.\cite{MikhalkinRau}), modeled after the cohomology ring of a toric variety \cite{FultonSturmfels97}, and plays an important role in combinatorial geometry as well. For example, the product of dual fans of polytopes in this ring is a 0-dimensional fan $[P_1]\cdot\cdots\cdot[P_n]$, whose weight records the mixed volume of the polytopes $P_1,\ldots,P_n$.

Algebraic geometry for tropical fans is an important common abstraction in geometry of algebraic varieties and of matroids, yielding important results in both fields.

We say that a $k$-dimensional tropical fan $F$ is \textit{regular} if there is a \textit{regular sequence} of polytopes $P_1,\dots,P_k$ such that $[P_1]\cdot\cdots\cdot[P_k]\cdot F=1$. This definition is modelled after the definition for a point $x$ being regular in an algebraic set $V\subset\mathbb{C}^n$, namely that there exists a local regular sequence defining a 0-dimensional complete intersection of multiplicity $1$ on $(V,x)$. Regular fans generalise the established class of tropical fans that are considered smooth, namely Bergman fans of matroids up to $\mathrm{SL}_n(\mathbb{Z})$-linear transformations \cite{Fink}. Smooth fans appear frequently in tropical geometry; which of their appearances can be generalised to regular fans?

Following a question in \cite{E19}, we aim to classify regular fans and regular sequences on them. 

There already exist some results that implicitly answer special cases of this question: if we require all polytopes in the regular sequence to equal the simplex, then all fans with this particular regular sequence are exactly matroid fans \cite{Fink}. If the codimension of $F$ is 1, then this is the known classification of polytopes of mixed volume 1 \cite{Esterov-Gusev}. Thus our goal can be seen as a search for a common generalization of these results.

A first idea of how the general answer might look like is given by the full answer for the 1-dimensional case.
\begin{theor}[Theorem 3.16]
1. Every one-dimensional fan $F$ is regular if and only if it admits a projection $\pi_F$ to a direct sum of 1-dimensional Bergman fans so that any other projection holding with this property factors through it.

2. Every regular sequence on $F$ lifts from a regular sequence on $\pi_F(F)$.
\end{theor}

The version for fans of arbitrary dimension has to be more complicated and resembles the minimal model programme.

\vspace{1ex}

{\sc ‘Minimal model programme' for regular cycles.} A generalized question of ‘minimal model programme' for tropical cycles from Theorem 1.1 is as follows:
\begin{quest}[A.Esterov]
     Does there exist a finite class of {\it minimal model} $k$-cycles, such that every regular $k$-cycle, satisfying the Hodge--Lefschetz package of \cite{AdiprasitoHuhKatz20},  is one of the following?

$\bullet$ Either it can be represented as a tropical fibration (Definition 2.15), from whose base and fibers all its regular sequences are induced,

$\bullet$ or it projects to a minimal model cycle, from which each of its regular sequence lifts. 
\end{quest}

In the second item of Question 1.2 the regular sequences should be interpreted as consisting of tropical rational functions defined only on the given regular $k$-cycle.

\begin{rem}
    Note that, instead of tropical fans, we work with affine tropical cycles, i.e., the equivalence classes of tropical fans modulo refinements (Definition 2.4). This is not only because working with tropical k-cycles is more general, but also because the finiteness theorem (Theorem 4.13) fails once we restrict attention to tropical fans of dimension two or higher instead of cycles.
\end{rem}

\vspace{1ex}

The Hodge--Lefschetz assumption in this question is satisfied by all of the most interesting classes of cycles like the aforementioned cycles of polytopes and matroids. Moreover, it serves as a useful tool which provides us strong inequalities on the tropical intersection number. For example, the Hodge–Lefschetz assumption reduces to the Hodge index theorem for $2$-cycles as in \cite{BabaeeHuh19}:
$$([P_1]\cdot[P_2]\cdot F)^2\ge([P_1]\cdot[P_1]\cdot F)([P_2]\cdot[P_2]\cdot F)$$
For a regular sequence $(M_1,M_2)$ on $F$, i.e., $M_1\cdot M_2\cdot [F]=1$ (Definition 2.6), this implies either

$$\ \ a).\ M_1\cdot M_1\cdot [F]=1;\ \ M_2\cdot M_2\cdot [F]=1\ \mathrm{or},$$
$$b).\ M_1\cdot M_1\cdot[F]=0\ \mathrm{or}\ M_2\cdot M_2\cdot[F]=0.$$

In future work, we expect to show that $b)$ imposes the structure of a regular fibration (Definition 2.15) on $F$, and now we completely resolve case $a)$ with the following theorem. 

\begin{theor}[Theorem 4.13]
     Assume that $M_1,M_2$ are two tropical regular functions over $\mathbb{R}^n$ such that dim$(\mathrm{Newt}(M_1)+\mathrm{Newt}(M_2))=n$. Then there are finitely many 2-cycles $[F]$ such that $M_1\cdot M_1\cdot [F]\le 1$, $M_1\cdot M_2\cdot [F]\le 1$, and $M_2\cdot M_2\cdot [F]\le 1$.
\end{theor}

When the Hodge index theorem hypothesis is removed, finiteness of regular $2$-cycles is not to be expected. However, one may hope to replace the Hodge–Lefschetz package in Question~1.2 by other, more general assumptions.

\begin{rem}
In future work, we expect to show that every cycle satisfying Theorem 1.4 has a matroid minimal model, thus confirming the MMP expectations.
\end{rem}
\vspace{1ex}
\begin{rem} 
The Hodge index theorem is satisfied by all of the most important classes of tropical fans, such as matroid fans (see \cite{AdiprasitoHuhKatz20}) and tropicalizations of algebraic varieties (see \cite{Esterov25}), so Theorem 1.4 applies to all regular sequences on all such fans $F$.

Fans failing to satisfy the Hodge--Lefschetz assumption do not look ubiquitous in algebraic or combinatorial geometry, or even easy to construct: for instance, the key achievement of the fundamental paper \cite{BabaeeHuh19} was to construct at least one such ``irreducible'' fan in $\R^4$, as it gave a counterexample to a certain strong version of the Hodge conjecture.

Still, it would be interesting to know whether one could replace the Hodge-theoretic assumptions with some milder alternatives in Question 1.2, so that Theorem 1.4 would hold true with these alternative assumptions instead of $a)$.
\end{rem}
\vspace{1ex}

\begin{rem} 
Besides supporting the conjecture, our classification is relevant to the study of so-called engineered complete intersections \cite{E24,Zhizhin,KAH25} -- a class of toric complete intersections which includes nondegenerate ones, hyperplane arrangement complements, generalized Calabi-Yau complete intersections, and several other important classes of algebraic varieties. For instance, in \cite{Esterov25}, the classification of reducible engineered complete intersections is related to a combinatorial question similar to the classification of regular cycles with the condition $[F]\cdot[P_1]\cdot\cdots\cdot [P_k]=0$ instead of $1$.

The classification of regular cycles will give more precise certificates for irreducibility of engineered complete intersections, in the same way as the classification of polytopes of unit mixed volume \cite{Esterov-Gusev} gave a precise classification of reducible nondegenerate complete intersections in \cite{KavehManon19}. Such irreducibility criteria are important e.g. when constructing special varieties with prescribed topology in the form of toric complete intersections which is the key source of e.g. Calabi--Yau varieties -- see e.g. \cite{Batyrev} for non-degenerate complete intersections. \\
\end{rem}

Although Section 2 introduces a few new notions—regular cycles (Definition 2.9) and
galleries (Definition 2.19)—most of its material is standard.
Readers already familiar with tropical geometry may treat this section as
preliminaries.

Section 3 gives a complete classification of regular $1$–cycles.
The first two subsections analyse their structure in detail (Theorem 3.1,3.2); Section 3.2
constructs the relevant regular sequences (Theorem 3.11), and Section 3.3 proves the
‘minimal model programme' for 1-dimensional tropical fans (Theorem 3.16).

Section 4 turns to dimension $2$. It presents a series of illustrative examples that highlight both key implications and their failures under our hypotheses.
After establishing the structures for $2$–dimensional galleries, we
prove a finiteness theorem for the strongly regular $2$–cycles (Theorem 4.13).\\

\paragraph*{\textbf{Acknowledgements.}}
The author is deeply grateful to \textsc{Alexander Esterov} and \textsc{Alex Fink} for their patient supervision, fruitful discussions, and valuable suggestions.  
This work was supported by the Chinese Scholarship Council.\\

\section{Regular cycles}

We introduce our main object in this paper—the regular cycles—after recalling the basic language of tropical fans, their intersection product \cite{AllermannRau10} and the stable intersection \cite{MaclaganSturmfels15}.  We also review Bergman fans \cite{ArdilaKlivans06} and introduce a new construct—the gallery—that will frame our later discussion of the cycles.\\ 

\subsection{Affine tropical cycles}
We will briefly review some fundamental objects for readers new to tropical geometry. Everything in this subsection is taken from \cite{AllermannRau10,Cox}.  
For additional examples, see Sections 1–3 of those references.
\begin{defin}[{\cite[Cones]{Cox}}]
A \textit{convex polyhedral rational cone} in $\mathbb{R}^n$ is a set of the form
$$\sigma=\mathrm{Cone}(S):=\{\sum_{u\in S}\lambda_uu\mid\lambda_u\ge 0\}\subseteq\mathbb{R}^n$$
where $S\subseteq\mathbb{Z}^n$ is finite. 

We denote the subspace spanned by $\sigma$ by $\langle\sigma\rangle$ and by $\mathbb{Z}_\sigma$ the lattice $\mathbb{Z}^n\cap\langle\sigma\rangle$. We call $\sigma$ strongly convex if and only if $\sigma\cap(-\sigma)=\{0\}$.
\end{defin}

\begin{defin}[{\cite[Weighted fans]{AllermannRau10,Cox}}]
    A fan $X$ in $\mathbb{R}^n$ is a finite set of cones in $\mathbb{R}^n$ satisfying the following conditions:

    (a) Every $\sigma\in X$ is a strongly convex rational cone in $\mathbb{R}^n$.

    (b) For all $\sigma\in X$, each face of $\sigma$ is also in $X$.

    (c) For all $\sigma_1,\sigma_2\in X$, the intersection $\sigma_1\cap\sigma_2$ is a  face of each.

    If $X$ is a fan, then

  1. the support of $X$ is $|X|=\bigcup_{\sigma\in X}\sigma\subseteq\mathbb{R}^n$.

  2. The dimension of $X$ is defined to be the maximum of the dimensions of the cones in $X$.

    3. The fan $X$ is called \textit{pure-dimensional} if each inclusion-maximal cone in $X$ has this dimension.

    We denote the set of all $k$-dimensional cones of $X$ by $X^{(k)}$.  

    A weighted fan $(X,\omega_X)$ of dimension $k$ is a fan $X$ of pure dimension $k$ equipped with a map $\omega_X:X^{(k)}\to\mathbb{Z}_{\ge0}$. The number $\omega_X(\sigma)$ is called the weight of the face $\sigma\in X^{(k)}$.
\end{defin}

\begin{defin}[Tropical fans]
    A tropical fan of dimension $k$ in $\mathbb{R}^n$ is a weighted fan $(F,\omega_F)$ of dimension $k$ satisfying the balancing condition for every $\tau\in F^{(k-1)}$:
    $$\sum_{\tau<\sigma}\omega(\sigma)v_{\sigma/\tau}=0\in \mathbb{R}^n/\langle\tau\rangle$$
    where $v_{\sigma/\tau}$ is the primitive generator from $\tau$ to $\sigma$ (\cite{MikhalkinRau} Section 2). We simply write $F$ instead of $(F,\omega_F)$ in this paper.
\end{defin}

Let $(X,\omega_X)$ and $(Y,\omega_Y)$ be two tropical fans. We will use the word `refinement' in the sense of \cite{AllermannRau10}; note that this differs from the way it is often used in polyhedral geometry. We call $(X,\omega_X)$ a refinement of $(Y,\omega_Y)$ if every cone of $X$ is a subset of a cone of $Y$ and the weight function $\omega_X$ is induced from $\omega_Y$. Two tropical fans are equivalent if they have a common refinement. Readers may consult \cite{AllermannRau10} for more details. 

\begin{defin}[Affine tropical cycles]
Let $(X,\omega_X)$ be a tropical fan of dimension $k$. We denote the equivalence class under the equivalence relation above by $[(X,\omega_X)]$. We refer to $Z^{\mathrm{aff}}_{k}(\mathbb{R}^n)$ as the set of equivalence classes of tropical fans over $\mathbb{R}^n$. The elements of $Z^{\mathrm{aff}}_k(\mathbb{R}^n)$ are called (affine tropical) k-cycles. Again, we write $[X]$ instead of $[(X,\omega_X)]$ if the weight function is clear.
\end{defin}

\subsection{Tropical intersection product and regular cycles}

\begin{defin}[Tropical regular functions] A tropical regular function $T(x):\mathbb{R}^n\to\mathbb{R}$ (TR function) is a global continuous piecewise linear convex function defined by
$$x\mapsto \max(l_1,\ldots,l_k)(x)$$
where $l_i\in(\mathbb{Z}^n)^\vee$. We will express $T$ as $\max(l_1,\dots,l_k)$ for convenience. In particular, any linear form $l_{i}$ that cannot be omitted is called a \textit{linear functional} of $T$. Moreover, we refer to $V(T)$ as the hypersurface of the TR function $T$.
\end{defin}

\begin{defin}[Tropical Intersection Product]
    Let $T$ be a TR function and let $[(F,\omega_F)]$ be a $k$-cycle with a representative $(F,\omega_F)$ such that $T$ is linear on each cone $\sigma\in F$.  The tropical intersection product is a $(k-1)$-cycle $T\cdot [F]:=[(\bigcup_{i=0}^{k-1}F^{(i)},\omega_{T\cdot[F]})]\in Z^{\mathrm{aff}}_{k-1}(\mathbb{R}^{n})$:
    \begin{align*}
        \omega_{T\cdot [F]}: F^{(k-1)} &\longrightarrow \mathbb{Z},\\
    \tau &\mapsto \sum_{\sigma\in F^{(k)},\tau<\sigma}T_\sigma(\omega_F(\sigma)v_{\sigma/\tau})-T_\tau(\sum_{\sigma\in F^{(k)},\tau<\sigma}\omega_F(\sigma)v_{\sigma/\tau})
    \end{align*}
    where $T_\sigma:\langle\sigma\rangle\to\mathbb{R}$ is the $\mathbb{Z}$-linear functional defined by $T_\sigma(x)=T(x)$ for all $x\in \langle\sigma\rangle$ (similar for $T_\tau$).
\end{defin}

Section 3 of \cite{AllermannRau10} shows that the definition is well-defined: it is independent of the choice of the normal vector $v_{\sigma/\tau}$ and of the representative $F$ of the cycle $[F]$.

\begin{lemma}[{\cite[Remark 3.6]{AllermannRau10}}]
    Let $T$ be a TR function over $\mathbb{R}^n$ and let $[F]$ be a $k$-cycle. Then $(T+L)\cdot[F]=T\cdot [F]$ for any integer linear function $L:\mathbb{Z}^n\to\mathbb{Z}$.
\end{lemma}

\begin{rem}
    Lemma 2.7 allows us to reduce our study from arbitrary tropical regular functions to non-negative tropical regular functions:
    $$\max(l_1,l_2,\ldots,l_k)\cdot [F]=(\max(l_1,l_2,\ldots,l_k)-l_i)\cdot [F]=\max(l_1-l_i,\ldots,0,\ldots,l_k-l_i)\cdot[F]$$
    for any linear functional $l_i$ of $T$.
\end{rem}

\begin{defin}[Regular cycles]
    An affine tropical $k$-cycle $[F]$ is called \textit{regular} if there exist non-negative TR functions $T_1,\ldots,T_k$ such that $T_1\cdot\cdots\cdot T_k\cdot[F]=1$. The corresponding tuple $(T_1,\ldots,T_k)$ is called a \textit{regular sequence}.
\end{defin}

\begin{rem}
    For a regular $k$–cycle $[F]$ the tropical complete intersection $$T_{1}\cdot\cdots\cdot T_{k}\cdot [F]$$ is the zero cycle of multiplicity~$1$. In classical algebraic geometry a $0$–dimensional scheme is called reduced when its multiplicity equals~$1$; in the same spirit, a regular $k$–cycle supports a reduced $0$–dimensional tropical complete intersection.
\end{rem}

\begin{defin}[Stable Intersection]
        Let $\Sigma_1,\Sigma_2$ be two pure weighted balanced polyhedral complexes in $\mathbb{R}^n$. We define the stable intersection between $\Sigma_1$ and $\Sigma_2$ to be
    $$\Sigma_1\cap_{st}\Sigma_2:=\bigcup_{\substack{\sigma_1\in\Sigma_1, \sigma_2\in\Sigma_2\\dim(\sigma_1)+dim(\sigma_2)=n}}\sigma_1 \cap \sigma_2,$$
    with multiplicity among the top-dimensional cells $\sigma_1\cap\sigma_2$:
    $$\omega_{\Sigma_1\cap_{st}\Sigma_2}(\sigma_1\cap\sigma_2)=\sum_{\tau_1,\tau_2}\omega_{\Sigma_1}(\tau_1)\omega_{\Sigma_2}(\tau_2)[\mathbb{Z}^n:\mathbb{Z}^n_{\tau_1}+\mathbb{Z}^n_{\tau_2}],$$
    where the sum is over all $\tau_1\in$ star$(\sigma_1\cap\sigma_2)_{\Sigma_1}$, $\tau_2\in$ star$(\sigma_1\cap\sigma_2)_{\Sigma_2}$ with $\tau_1\cap(\textbf{v}+\tau_2)\neq\emptyset$ for some generic $\textbf{v}$. Lemma 3.6.4 of \cite{MaclaganSturmfels15}  shows that it is independent of choice of $\textbf{v}$.
\end{defin}

The next theorem is a central ingredient of the paper: it shows that stable intersection coincides with the tropical intersection product.

\begin{theor}[{\cite[Theorem 4.4]{Katz09} }]
    Let $T$ be a TR function and let $[F]$ be a $k$-cycle. Then the tropical intersection product admits tropical stable intersection:
    $$T\cdot [F]=[V(T)\cap_{st} F]$$
    for any representative $F$ of $[F]$.\\
\end{theor}

\begin{rem}[TR functions vs. Polytopes]
    Theorem 2.12 is a well-known result that serves as a nice bridge between the tropical intersection theory and the polyhedral geometrical intersection. Let $T$ be a non-negative TR function defined by $\max(0,l_1^\vee,\ldots,l_k^\vee)$ over $\mathbb{R}^n$ and let $P$ be a polytope defined by $\mathrm{conv}(0,l_1,\ldots,l_k)$. Note that $T$ is the support function of $P$ and we define $[P]$ to be the dual fan of $P$. Then for an arbitrary $k$-cycle $[F]\in Z^{\mathrm{aff}}_k(\mathbb{R}^n)$, we have
    $$T\cdot [F]=[V(T)\cap_{st} F]=[[P]\cap_{st} F]=:[P]\cdot [F].$$
    
\end{rem}

The following lemma will be applied frequently.

\begin{lemma}[{\cite[Projection formula]{AllermannRau10}}]
   Let $C\in Z^{\mathrm{aff}}_k(\mathbb{R}^n)$ be an affine $k$-cycle and let $f:\mathbb{R}^n\to\mathbb{R}^m$ be an integer linear map such that $f_*(C)\in Z^{\mathrm{aff}}_k(\mathbb{R}^m)$. Let $T$ be a TR function over $\mathbb{R}^m$. Then the following holds:
   $$T\cdot(f_*C)=f_*(f^*T\cdot C).$$
\end{lemma}

We include the following definition, which is not used in our text, but explains the precise statement of Question 1.2 motivating our work.

\begin{defin}[Regular fibration]
    A projection $p:\R^n\to\R^m$ defines a {\it regular fibration} of a tropical fan $F\subset\R^n$ over its set-theoretical image $G:=p(|F|)$, if 

1) $G$ has a structure of a regular tropical fan (which defines the central fiber $C$ of $p|_F$ as a tropical fan),

2) The tropical fan $C$ is regular.
\end{defin}

\begin{rem}
   Any regular $k$-cycle $[F]$ over $\mathbb{R}^n$ gives a regular fibration of a tropical fan $\tilde{F}$ over $\mathbb{R}^{n+1}$. For example, let $T_1,\ldots,T_k$ be a regular sequence of $[F]$ over $\mathbb{R}^n$ and let $\psi$ be an arbitrary TR function. Then \cite{Shaw13} gives a $k$-cycle $\tilde{F}$ over $\mathbb{R}^{n+1}$ such that
   $$\psi\cdot [F]=[\tilde{F}\cap_{st}V(\max(0,e_{n+1}))].$$
   Therefore, we get
   $$\delta^*T_1\cdot\cdots\cdot\delta^*T_{k}\cdot\tilde{F}=T_1\cdot\cdots\cdot T_{k}\cdot\delta_*(\tilde{F})=T_1\cdot\cdots\cdot T_k\cdot F=1$$
   where $\delta:\mathbb{R}^{n+1}\to\mathbb{R}^{n+1}/e_{n+1}$ sending $\tilde{F}$ to $F$ is called the elementary modification (see \cite[Definition 2.10]{Shaw13}). Although the elementary modification defines a general fibration with base $F$ and a central fibre $\langle e_{n+1}\rangle$ (they are even regular), $\delta$ is not a regular fibration. This is because we require $p(|F|)$ to be equipped with a structure of a tropical fan whose dimension is strictly smaller than $F$ if $p$ is a regular fibration of $F$.
\end{rem}

    

Although Bergman fans will play only a minor role in this paper—they are needed explicitly only in  Theorem 1.1—we briefly recall the definition and point the reader to standard references for details.  Background material can be found in
\cite{Fink,ArdilaKlivans06}, where the combinatorial and geometric properties of Bergman fans (also called {tropical linear spaces}) are developed in depth.

\begin{defin}
    Let $M=(E,\mathcal{B})$ be a loopfree matroid of rank r$(M)$ with ground set $E=[n]$. The Bergman fan $\mathcal{B}(M)$ is the fan
    $$\mathcal{B}(M):=\{w\in\mathbb{R}^n\mid\ M_w\ \rm{loopfree}\}$$
    where $M_w$ is the associated matroid whose bases are the $w$-minimum bases of $M$.
\end{defin}

Let $N$ be a lattice and let $[F]$ be a tropical $k$-cycle in $N_{\mathbb R}$.
We say that $[F]$ is the \textit{direct sum} of two $k$-cycles $[X]$ and $[Y]$ if $[F]=[X]+[Y]$ in the sense of the addition in the group $Z^{\mathrm{aff}}_k(N_\mathbb{R})$ (see \cite{AllermannRau10}) and $N = N_X \oplus N_Y$, where $N_X$ (respectively $N_Y$) is the minimal sublattice of $N_{\mathbb R}$ whose ambient space contains $X$ (respectively $Y$).

\begin{lemma}
     Assume $X=\bigoplus_{i=1}^r B_i$ to be a direct sum of 1-D Bergman fans $B_i$ over $\mathbb{R}^n$. Then the TR function $M$ satisfies $M\cdot [X]=1$ if and only if $\mathrm{Newt}(M)$ can be shifted to the standard simplex in one of the spaces $\langle B_i\rangle$.
\end{lemma}
\begin{proof}
    Note that
$$M\cdot[X]=\sum_{i=1}^q\omega_iM(r_i)=\sum_{j=1}^rM\cdot B_j=1$$
    where $\{r_1,\ldots,r_q\}$ is the collection of primitive generators of $[X]$. Lemma 2.7 yields that each term $M\cdot B_j\ge 0$. Therefore there exists a unique 1-D Bergman fan $B_k$ whose intersection number is non-zero. Then it follows direct from Theorem 6.5 and Definition 3.3 of \cite{Fink}.
\end{proof}

Finally, let's introduce a combinatorial structure of tropical fans and cycles called a \textit{gallery} which serves as a key object in our story. We refer to a piecewise linear function of the form $\max(l,h)$ for some $l,h\in(\mathbb{Z}^n)^\vee$ as a tropical binomial.

\begin{defin}[gallery of a fan]
Let $[F]$ be a regular $k$-cycle in $\mathbb{R}^n$ and let $(L_1,\ldots,L_k)$ be a regular sequence of $[F]$ such that $L_i$'s are tropical binomials for all $i\in\{1,\ldots,k\}$. Then for any representative $F$ of $[F]$, the {\it gallery} of $F$ over $\{L_1,\dots,L_k\}$ is the set of faces:
\begin{center}
    $g(F,L_1,\dots,L_k):=\{\sigma\in F^{(k)}\mid L_1\cdot\cdots\cdot L_k\cdot[\sigma]=1\}.$
\end{center}

We denote the set of all galleries of $F$ by $\rm{Gly}(F)$.
\end{defin}

The following lemma can be immediately obtained from the polyhedral geometry of fans:

\begin{lemma}
    Let $L_1\cdot\cdots\cdot L_k\cdot[F]=1$ where $L_i$'s are tropical binomials and let $F,G\in[F]$ be two representatives of $[F]$. Then the supports of the galleries of $F,G$ over $\{L_1,\ldots,L_k\}$ coincide.
\end{lemma}

\begin{defin}[gallery of a cycle]
    Let $[F]$ be a regular $k$-cycle in $\mathbb{R}^n$ and let $(L_1,\ldots,L_k)$ be a regular sequence of $[F]$ such that $L_i$'s are tropical binomials for all $i\in\{1,\ldots,k\}$. Then the gallery of $[F]$ over $\{L_1,\ldots,L_k\}$ is the support of $g(F,L_1,\ldots,L_k)$ for any representative $F\in[F]$. Similarly, we denote the set of all galleries of $[F]$ by $\mathrm{Gly}([F])$.\\
\end{defin}

\section{Full classification for regular $1$-cycles}

One-dimensional tropical cycles are the basic building blocks of tropical geometry.  
They form the simplest non-trivial class of tropical cycles and already capture many key features of tropical intersection theory and matroid theory.  
Let $M$ be a non-negative tropical regular (TR) function on $\mathbb Z^{n}$ and let $F\subset\mathbb R^{n}$ be a $1$-cycle.

A $1$-cycle has a single representative with positive weights, so it is the same object as a one-dimensional tropical fan.  
Throughout this section we work exclusively with one-dimensional tropical fans, not with $1$-cycles, in order to highlight the fan structure.

Throughout, we assume that
$$
  F \;=\;
  \bigl(\{\mathbf 0,\;\mathbb R_{+}v_{1},\ldots,\mathbb R_{+}v_{q}\},\,
         \omega_{i}\bigr)
$$
is a one-dimensional tropical fan in $\mathbb R^{n}$ with primitive generators $v_{1},\ldots,v_{q}$ and positive weights $\omega_{i}$. Moreover, we also assume that the ambient space $\langle F\rangle$ is exactly $\mathbb{R}^n$.

The goal of this section is to classify all pairs $(M,F)$ satisfying
$$
  M\!\cdot\!F \;=\; 1
$$
such that $M$ is a non-negative tropical regular function and $F$ is the 1-dimensional tropical fan introduced above.

We state two main theorems for a 1-dimensional regular tropical fan $F$.  
The gallery-existence theorem gives a coordinate description of $F$, and the ‘minimal-model programme' for one-dimensional fans explains the relationship between $F$ and Bergman fans via a universal property.  
We also describe the full set of non-negative TR functions $M$ whose tropical intersection product with $F$ equals $1$.

The following theorem is the first basic result toward the classification of one-dimensional regular tropical fans.

\begin{theor}[Gallery-existence theorem]
   If there exists a non-negative TR function $M$ such that $M\cdot F=1$, then $\mathrm{Gly}(F)\neq\emptyset$.  Moreover, there is a suitable coordinate system such that fan $F$ has the form:

    $v_1=(1,0,\dots,0)$, $\omega_1=1$;

    $v_2=(-1,0,\dots,0,m)$, $\omega_2=1$;

    $v_i\in (1,0,\dots,0)^\perp$.
\end{theor}

\begin{proof}
    The complete intersection $1$ yields $\sum_{i=1}^q\omega_iM(v_i)=1$, and the non-negativity of each part $\omega_iM(v_i)$ implies that there exists a unique index (say $v_1$) such that $\omega_1M(v_1)=1$ and $\omega_iM(v_i)=0$ for $i>1$.

    On the one hand, $\omega_1M(v_1)=1$ tells us that $\omega_1=1$ and $M(v_1)=1$. Thus there exists at least one binomial $\max(0,l)\le M$ such that $l(v_1)=1$.

    On the other hand, balancing condition of $F$ implies
    $0=\sum_{i=1}^q\omega_iv_i.$
    So we have
    $$l(\sum_{i=1}^q\omega_iv_i)=\sum_{i=1}^q\omega_il(v_i)=0=\omega_1l(v_1)+\sum_{j=2}^q\omega_jl(v_j).$$
    Therefore $-1=\sum_{j=2}^q\omega_jl(v_j)$. Note that $l(v_j)\le M(v_j)=0$ and $\omega_j\ge 1$, hence there exists a unique index (say $v_2$) such that $\omega_2l(v_2)=-1$ and $\omega_jl(v_j)=0$ for all $j>2$. Therefore, the set $\{\mathbb{R}_+\cdot v_1,\mathbb{R}_+\cdot v_2\}$ is a set of gallery over $\max(0,l)$ as required.

    Finally, let's prove the desired coordinate exists. Notice that $l:\mathbb{Z}^n\to\mathbb{Z}$ is a $\mathbb{Z}$ linear function, and the image $\mathrm{Im}(l)=\mathrm{span}_\mathbb{Z}(l(v_1))\cong \mathbb{Z}$ which is a free $\mathbb{Z}$ module. Hence $\mathbb{Z}^n/\mathrm{ker}(l)\cong \mathrm{Im}(l)\cong \mathbb{Z}$ gives that $\mathrm{ker}(l)$ is a rank $n-1$ free $\mathbb{Z}$ module. Now let $w_2,\dots,w_n$ be an $\mathbb{Z}$-basis of $\mathrm{ker}(l)$, then the integer matrix $T:=[l,w_2,\dots,w_n]^T\in \mathrm{GL}(n,\mathbb{Z})$. Therefore, $T(v_1)=(1,0,\dots,0)$ with weight $1$ as required. The coordinate statement of $v_2$
    can be directly obtained by applying a $T'\in \mathrm{GL}(n,\mathbb{Z})$ on $\mathrm{ker}(l)$. Since $\omega_jl(v_j)=0$ for all $j>2$, this gives us that $v_j\in \mathrm{ker}(l)=x_1^\perp$ for all $j>2$ as required.\\
\end{proof}

\subsection{Canonical partition}

Our key target of this subsection is to show the following theorem:

\begin{theor}
    Consider the relation $\sim$ on the collection of rays $\{\mathbb{R}_+\cdot v_1,\ldots,\mathbb{R}_+\cdot v_q\}$ such that $\mathbb{R}_+\cdot v_i\sim\mathbb{R}_+\cdot v_j$ if and only if $i=j$ or $\{\mathbb{R}_+\cdot v_i,\mathbb{R}_+\cdot v_j\}\in\mathrm{Gly}(F)$. Then $\sim$ is an equivalence relation.
\end{theor}

The following lemma can also be viewed as a definition of a 1-dimensional gallery. Recall that a gallery $g(F,L)$ can be explicitly defined by a linear function $l\in(\mathbb{Z}^n)^\vee$ such that $V(L)=l^\perp$.  

\begin{lemma}
   $g(F,L)$ is the set of two rays $\{\mathbb{R}_+\cdot v_a,\mathbb{R}_+\cdot v_b\}$ and the following conditions hold:

    $\bullet$ $l(v_a)=1$, $l(v_b)=-1$ and $\omega_a=\omega_b=1$;\

    $\bullet$ $l(v_i)=0$ for other generators $i\neq a,b$.
\end{lemma}

\begin{proof}
    By Definition 2.19, $g(F,L)$ is the following set
    $$\{r\in F\mid L\cdot[r]=1\}.$$
    Let $L:=\max(0,l)$. Then we have $\max(0,l)(r)+\max(0,l)(-r)=1$ if and only if $r\in g(F,L)$. On the one hand, we have either $l(r)=1$ or $l(-r)=1$ for any $r\in g(F,L)$. On the other hand, the proof of Theorem 3.1 implies that there exists a unique pair of rays $\{\mathbb R_+\cdot v_a,\mathbb R_+\cdot v_b\}$ such that $l(v_a)=1$, $l(v_b)=-1$, $\omega_a=1=\omega_b$ and $l(v_i)=0$ for any other $i\neq a,b$ since $\max(0,l)\cdot F=1$. Combine these together, we conclude that $\{\mathbb R_+\cdot v_a,\mathbb R_+\cdot v_b\}$ is the gallery $g(F,L)$ as we required.
\end{proof}

Recall that we denote the set of all galleries in the fan $F$ by $\mathrm{Gly}(F)$. We adapt the following notation $g_l(v_a,v_b)$ instead of $g(F,L)$, in order for the reader to clearly see all the information recorded in the gallery:
\begin{center}
    $g_l(v_a,v_b)\in\mathrm{Gly}(F)$ with $l(v_a)=1$, $l(v_b)=-1$.
\end{center}
\begin{rem}
    The notation $g_l(v_a,v_b)$ is independent of the order of $\{v_a,v_b\}$ up to sign: If $g_l(v_a,v_b)\in\mathrm{Gly}(F)$, then 
    \begin{center}
    $g_{-l}(v_b,v_a)\in\mathrm{Gly}(F).$\\ 
    \end{center}
   
\end{rem}

Now we are going to show that $F$ has a canonical partition with respect to the galleries of $F$.

\begin{lemma}
    If $g_l(u_1,u_2),g_h(u_1,u_3)\in\mathrm{Gly}(F)$ and $u_2\neq u_3$, then $g_{f}(u_2,u_3)\in\mathrm{Gly}(F)$ for some $f\in(\mathbb{Z}^n)^\vee$.
\end{lemma}

\begin{proof}
    By definition, $l(u_1)=h(u_1)=1$, we set $f=h-l\in(\mathbb{Z}^n)^\vee$.  Then we have $(h-l)(u_2)=h(u_2)-l(u_2)=0-(-1)=1$, $(h-l)(u_3)=h(u_3)=-1$ and $(h-l)(u)=0$ for all generators $u\neq u_2,u_3$. Clearly $\omega_2=\omega_3=1$, hence the Lemma 3.3 holds. So $g_f(u_2,u_3)\in\mathrm{Gly}(F)$ as required.
\end{proof}

\begin{proof}[proof of Theorem 3.2]
   The reflexivity and symmetry come from the definitions directly, and the transitivity is given by Lemma 3.5.
\end{proof}

\begin{defin}
 We set $P_i$ to be the equivalence class of $\mathbb R_+\cdot v_i$ under the equivalence relation defined in Theorem 3.2.
\end{defin}

\begin{defin}[Canonical partition]
    We refer to a canonical partition of $F$ as the subset \textbf{P}$\subseteq\{1,2,\ldots,q\}$ such that 

    1. $i\in\textbf{P}$ if $P_i$ contains at least two elements.

    2. $P_i\cap P_j=\emptyset$ for all $i,j\in\textbf{P}$.
\end{defin}

As a consequence, we can split $F$ as
\begin{center}
 $\{$rays of $F\}$= $\{$non-gallery rays$\}$ $\dot\cup$ $\dot\bigcup_{i\in\mathbf{P}}P_i$.\\
\end{center}

\subsection{Tropical regular functions}

We will give a classification of all possible non-negative tropical regular functions $M$ such that $M\cdot F=1$.

\begin{defin}
    Consider the canonical partition $\mathbf{P}$ of $F$. For any $i\in\mathbf{P}$ we define $\overline{\mathscr{M}_i}$ to be the set of non-negative tropical regular functions $M$:
    $$\overline{\mathscr{M}_i}:=\{M\mid M\cdot F=1,\exists \mathbb
    {R}_+\cdot u\in P_i\ s.t.\ M(u)=1\}.$$
\end{defin}

We claim that
\begin{center}
   $\{$non-negative TR functions $M$ such that $M\cdot F=1\}$ $=$ $\bigcup_{i\in\mathbf{P}}\overline{\mathscr{M}_i}$.
\end{center}

Indeed, this is because any non-negative TR function $M$ reveals at least one gallery as Theorem 3.1 discussed. The support of this gallery belongs to one equivalence class $P_i$, which gives $M\in\overline{\mathscr{M}_i}$.\\

 The following subset of $\overline{\mathscr{M}_i}$ contains all significant information from $\overline{\mathscr{M}_i}$.
\begin{defin}
    For any $i\in\textbf{P}$ and $\mathbb{R}_+\cdot u\in P_i$, a subset $\mathscr{M}_i(u)$ of $\overline{\mathscr{M}_i}$ is defined by
    \begin{center}
 $\mathscr{M}_i(u):=\{M\mid M\cdot F=1,\ M(u)=1\}.$       
    \end{center}
\end{defin}

By the definitions above, it is clear that
$$\overline{\mathscr{M}_i}=\bigcup_{u:\ \mathbb{R}_+u\in P_i}\mathscr{M}_i(u).$$

Recall that the ambient space of $F$ is $\mathbb{R}^n$. The following lemma suggests a uniqueness of the tropical binomial over a gallery of $F$.
\begin{lemma}
    Let $g_l(v_i,v_j)=g_{h}(v_i,v_j)$ be a gallery of $F$. Then $l=h$.
\end{lemma}

\begin{proof}
    Theorem 3.1 allows us to assume $v_i:=(1,0,\dots,0)^T$, $v_j:=(-1,0,\dots,q)^T$ and $v_k\in (1,0,\dots,0)^\perp$ for all $k\neq i,j$. Therefore, $v_i$ is part of an $\mathbb{Z}$-basis $\{v_i,w_1,w_2,\dots,w_{n-1}\}$ ($w_i\in (1,0,\ldots,0)^\perp$ for all $1\le i\le n-1$), and the linear functional $l$ satisfying $l(v_i)=1$, $l(v_j)=-1$, $l(v_k)=0$ is exactly the solution of
    $$[v_i,w_1,w_2,\dots,w_{n-1}]^T\textbf{x}=(1,0,\dots,0)^T.$$
    The uniqueness (and existence) of such linear functional comes from det$[v_i,w_1,\dots,w_{n-1}]=\pm 1$.
\end{proof}

\begin{theor}
    Fix an index $a\in\textbf{P}$, we assume $P_a=\bigcup_{i=0}^{|P_a|}\mathbb{R}_+\cdot v_i$ and denote $l_{ij}$ to be the unique functional $l_{ij}\in(\mathbb{Z}^n)^\vee$ such that $g_{l_{ij}}(v_i,v_j)\in P_a$. Then
    \begin{center}
        $\mathscr{M}_a(v_i):=\{\text{non-negative TR functions }M\mid0< M\le M_{max}(i)\}$
    \end{center}
    where $M_{max}(i):=\max(0,l_{i0},l_{i1},\dots,l_{i|P_a|})$ for any $i\in\{0,1,2,\dots,|P_a|\}$. Note that $\mathscr{M}_a(v_i)$ contains $|P_a|$ linear functionals since $l_{ii}$ is not well-defined.
\end{theor}

\begin{proof}
    $M_{max}(i)\in\mathscr{M}_a(v_i)$ since
    $$M_{max}(i)\cdot F=\sum_{j=1}^q\omega_jM_{max}(i)(v_j)=M_{max}(i)v_i=1.$$
        Now let's prove that $M_{max}(i)$ is exactly the maximal function among the $\mathscr{M}_a(v_i)$. By contradiction, if there exists a non-negative TR function $M'\in\mathscr{M}_a(v_i)$ for all $j\neq i$. such that $M'\not\subseteq M_{max}(i)$. Then we are allowed to assume $M_{max}(i)<M$. This is because the non-negative TR function $\max(M_{max}(i),M')$ is a regular sequence of $F$. To see this,
        $$\max(M_{max}(i),M')\cdot F=\sum_{j=1}^q\omega_j\max(M_{max}(i),M')(v_j)=\max(M_{max}(i),M')(v_i)=1$$
        since $\max(M_{max}(i),M')(v_j)=\max(M_{max}(i)(v_j),M'(v_j))=0 $. Therefore, the non-negative TR function $\max(M_{max}(i),M')$ is strictly larger than $M_{max}(i)$, so we can replace $M'$ by this $\max(M_{max}(i),M')$ as required. Hence we set such $M'\in\mathscr{M}_{a}(v_i)$ satisfying $M_{max}(i)<M'$.
        Then $\mathrm{Newt}(M_{max}(i))\subset \mathrm{Newt}(M')$. So we set $M'$ to be
    $$M':=max(0,l_{i0},l_{i1},\dots,l_{i|P_a|},l_1',l_2',\dots,l_k').$$
    This implies that $l'_j(v_i)\le 0$ for all $1\le j\le k$. Otherwise, say $l'_1(v_i)=1$. Then there exists a $\rho\neq v_i$ such that $l'_1(\rho)=-1$ and $l'(r)=0$ for any other $r\neq \rho,v_i$. This tells us that $g_{l'_1}(v_i,\rho)\in\mathrm{Gly}(F)$. Therefore, the canonical partition forces that $g_{l'_1}(v_i,\rho)\subseteq P_a$, so $\rho=v_j$ for some $j\neq i$. Lemma 3.10 ensures $l'_1$ equal to one of $\{l_{i0},l_{i1},\dots,l_{i|P_a|}\}$.

    Furthermore, $l_j'(r)\le 0$ for all $r\neq v_i$, otherwise say $l_j'(r')>0$, then $M'\cdot F\ge l_{i1}(v_i)+l_j'(r')>1$, contradiction. Balancing condition tells us that cone$(v_1,v_2,\dots,v_q)\cong \mathbb{R}^n$. So $l_j'(\mathbb{R}^n)\le 0$. Hence this holds only for $l_j'=0$. Therefore, $M_{max}(i)=M'$ as required.

    Finally, for any $0<M<M_{max}(i)$, $M(v_i)=1$ and $M(r)=0$ for all $r\neq v_i$. Hence $M\in\mathscr{M}_a(v_i)$ as required.
\end{proof}

\begin{coro}
    The linear functionals $l_{i0},\ldots,l_{i|P_a|}$ are linearly independent. Moreover, let $\pi:\mathbb Z^n\to\mathbb Z^{|P_a|}$ defined by the integer matrix $[l_{i0},l_{i1},\cdots,l_{i|P_a|}]^T$. Then $|\mathbb Z^{|P_a|}/\pi(\mathbb Z^n)|=1$.
\end{coro}

\begin{proof}
    Note that $\pi$ sends $v_i$ to $(1,0,\dots,0)\in\mathbb Z^{|P_a|}$, $v_1$ to $(-1,0,\dots,0)$ and $v_{|P_a|}$ to $(0,\dots,0,-1)$ as so on. Therefore, the image $\pi(\mathbb Z^n)$ is exactly the lattice $\mathbb Z^{|P_a|}$ and the independency holds immediately.
\end{proof}

\begin{rem}
    We close this subsection by concluding how to find all non-negative TR functions $M$ with $M\cdot F=1$ for a fixed 1-dimensional tropical fan: Given such fan $F$, we can read the canonical partition $\mathbf{P}$ of $F$ by Section 3.1; then we obtain all the $\mathscr{M}_a(v_i)$ for any $a\in\mathbf{P}$ and $\mathbb R_+\cdot v_i\in P_a$ by Theorem 3.11. The discussion of Definition 3.9 implies an explicit way to get $\overline{\mathscr{M}_a}$ from $\mathscr{M}_a(v_i)$ and finally
    \begin{center}
   $\{$non-negative TR functions $M$ such that $M\cdot F=1\}$ $=$ $\bigcup_{a\in\mathbf{P}}\overline{\mathscr{M}_a}$
\end{center}
gives all the possible non-negative TR functions $M$ with $M\cdot F=1$.\\
\end{rem}

\subsection{‘Minimal model programme' for 1-cycles}

We will prove Theorem 1.1 from the introduction in this subsection.
Given a canonical partition $\mathbf{P}$ on $F$, the ‘minimal model programme' for 1-dimensional fans will tell us that: each $P_i$ is mapped to a Bergman fan. Moreover, there exists a projection map sending $F$ to a direct sum of Bergman fans and any other projection like this is a factor through this one.

\begin{lemma}
    Let $\pi:\mathbb{Z}^n\to\mathbb{Z}^k$ be a linear map such that $\pi_*(F)$ is a 1-dimensional tropical fan. Then we call the fibre of $\pi_*(F)$ the set of rays of $F$ mapped to a ray of $\pi_*(F)$. If $\pi_*(F)$ is a $1$-D Bergman fan over $\mathbb R^k$, then the fibre of $\pi_*(F)$ is contained in some $P_i$ of $F$.
\end{lemma}

\begin{proof}
    Let's assume the generators of Bergman fan $\pi_*(F)$ have coordinates $r_i:=(0,\dots,1_i,\dots,0)$ for $1\le i\le k$ and $r_0=(-1,-1,\dots,-1)$ where $1_i$ means $1$ lives in the $i$-th place. Consider a functional $l_i:\mathbb{Z}^k\to\mathbb{Z}$ with coordinate $(0,\dots,-1_i,\dots,0)\in(\mathbb{Z}^k)^\vee$. It is clear that $g_{l_i}(r_0,r_i)\in\mathrm{Gly}(\pi_*(F))$. Therefore, projection formula gives
    $$1=\max(0,l_i)\cdot(\pi_*(F))=\pi_*(\pi^*(\max(0,l_i))\cdot F)=\pi^*(\max(0,l_i))\cdot F.$$
    The last equality holds from $|\mathbb Z^k/\pi(\mathbb Z^n)|=1$ since $\pi_*(F)$ is a Bergman fan. Therefore,
    $$g_{\pi^*(l_i)}(\pi^{-1}(r_0),\pi^{-1}(r_i))\in\mathrm{Gly}(F).$$
    Finally, Lemma 3.5 implies that
    \begin{center}
        fibre of $\pi_*(F)$ $=$ $\bigcup_{i=0}^k\mathbb{R}_+(\pi^{-1}(r_i))$ $\subseteq$ $P_0$ 
    \end{center}
        where $P_0:=\{$ $g\in\mathrm{Gly}(F)$ $|$ $\mathbb{R}_+\cdot\pi^{-1}(r_0)\in g$ $\}.$
    
\end{proof}

\begin{lemma}
    Let's fix a partition $\mathbf{P}$. For any $i\in\mathbf{P}$, there exists a $\mathbb{Z}$-linear map $\pi_i:\mathbb{Z}^n\to\mathbb{Z}^{|P_i|}$ such that ${\pi_i}_*(F)$ is isomorphic to a $1$-D Bergman fan over $\mathbb{Z}^{|P_i|}$.
\end{lemma} 

\begin{proof}
    Assume $1\in\mathbf{P}$ such that $|P_1|=k$, i.e., we may assume $P_1:=\{g_{l_i}(v_i,v_1)\mid2\le i\le k+1\}$ after changing the indices. Now we define $\pi_1:\mathbb{Z}^n\to\mathbb{Z}^k$ by the $\mathbb{Z}_{k\times n}$ matrix
    $$[l_2,l_3,\dots,l_{k+1}]^T_{k\times n}.$$
    Recall that for any $2\le i\le k+1$, $l_i$ sends $v_{i}$ to $1$, $v_1$ to $-1$ and $v_j$ to $0$ for all $j>k+1$. So it follows that
    \begin{align*}
        \pi_1(v_{i}) &=(0,\dots,1_{i-1},\dots,0)\in\mathbb{Z}^k\ \ 2\le i\le k+1;\\
        \pi_1(v_j) &=0\in\mathbb{Z}^k\ \ j>k+1;\\
        \pi_1(v_1) &=(-1,-1,\dots,-1)\in\mathbb{Z}^k.
        \end{align*}
        Note that $\mathrm{span}_{\mathbb{Z}}(\pi_1(v_{i})\mid1\le i\le k)\cong \mathbb{Z}^k=\mathbb{Z}^{|P_1|}$, so $\pi_1$ is exactly a $\mathbb{Z}$-linear map. \
        
        Finally, $\pi_1(v_{i})$ and $\pi_1(v_1)$ are primitive generators, so $\pi_1$ preserves their weights. In other words, the weights of $\mathbb{R}_+\cdot\pi_1(v_1)$ and $\mathbb{R}_+ \cdot\pi_1(v_{i})$ (for $2\le i\le k+1$) are all $1$. Therefore, ${\pi_1}_*(F)$ is exactly a Bergman fan over $\mathbb{Z}^{|P_1|}$. This holds for all $i\in\mathbf{P}$, hence we are done.
\end{proof}

Now we have our desired structure theorem. The conjecture that led to this theorem was originally proposed by Alex Esterov. Recall that Lemma 2.18 gives a list of all non-negative TR functions which serve as regular sequences of a direct sum of 1-D Bergman fans over $\mathbb{R}^n$.

\begin{theor}[‘Minimal model programme' for 1-cycles]
   1. Let $F$ be a $1$-dimensional tropical fan over $\mathbb{R}^n$. $F$ is regular if and only if there is a $\mathbb{Z}$-linear map $\pi_F$ sending $F$ to a direct sum of finitely many $1$-dimensional Bergman fans.

   2. If $\pi:\mathbb{R}^n\to\mathbb{R}^r$ is a $\mathbb{Z}$-linear map such that $\pi_*(F)$ is also a direct sum of $1$-D Bergman fans, then $\pi$ is a factor of $\pi_F$.

   3. Every non-negative TR function $M:\mathbb{R}^n\to\mathbb{R}$ with $M\cdot F=1$ lifts from a non-negative TR function $M':\mathbb{R}^{d}\to\mathbb{R}$ such that $M'\cdot {\pi_F}_*(F)=1$ and $d=$dim$(\langle{\pi_F}_*(F)\rangle)$.
\end{theor}

\begin{proof}

1. The \textit{if} direction comes from the projection formula directly. Now let's prove the \textit{only if} part.

Consider the canonical partition of $F$ with $\mathbf{P}$. We may assume $\mathbf{P}=\{1,2,\dots,k\}\subseteq\{1,2,\dots,q\}$ after re-assignment. We denote them by
$$P_i:=\{g_{l^i_j}(v^i_j,v^i_0)\mid1\le j\le |P_i|   \}.$$
Now we define a projection $\pi_F:\mathbb{R}^n\to \mathbb{R}^{\sum_{i\in\mathbf{P}}|P_i|}$ given by the following matrix
$$[l^1_1,\dots,l^1_{|P_1|},l^2_1,\dots,l^2_{|P_2|},\dots,\dots,l^k_{1},\dots,l^k_{|P_k|}]^T\in M_{(\sum_{i\in\mathbf{P}}|P_i|)\times n}(\mathbb{Z}).$$
The proof of Lemma 3.15 derives that $\pi_F$ sends $P_i$ to a Bergman fan over $\mathbb{Z}^{|P_i|}$ embedded in $\mathbb{Z}^{\sum_{i\in\mathbf{P}}|P_i|}$. Therefore, $$span_{\mathbb{Z}}(\pi(v^i_j)\mid1\le i\le k;\ 1\le j\le |P_i|)\cong \mathbb{Z}^{\sum_{i\in\mathbf{P}}|P_i|}.$$
Hence $\pi_F:\mathbb{Z}^n\to\mathbb{Z}^{\sum_{i\in\mathbf{P}}|P_i|}$, so it is exactly a $\mathbb{Z}$-linear map. Finally, it is clear that ${\pi_F}_*(F)$ is a direct sum of Bergman fans implied by each $P_i$, so ${\pi_F}_*(F)$ is a direct sum of $k$ Bergman fans as required.\\

2. Assume $\pi:\mathbb{R}^n\to\mathbb{R}^r$ to be a $\mathbb{Z}$-linear map sending $F$ to a direct sum of Bergman fans: $\pi_*(F)=\oplus_{i\in\textbf{A}}B_i$ for an index set $\textbf{A}$. We deduce from Lemma 3.14 that $r\le \sum_{i\in\textbf{P}}|P_i|$. So it is enough to show that there exists a $\mathbb{Z}$-linear map $\psi:\mathbb{Z}^{\sum_{i\in\textbf{P}}|P_i|}\to\mathbb{Z}^r$ such that following diagram commutes:
\begin{center}
\begin{tikzcd}[column sep=large,row sep=large]
  \mathbb{Z}^{n}
    \arrow[r, "\pi_{F}"]
    \arrow[dr, "\pi"']
  & \displaystyle \mathbb{Z}^{\sum_{i\in \textbf{P}}\lvert P_{i}\rvert}
    \arrow[d, "\psi"] \\
  & \mathbb{Z}^{r}
\end{tikzcd}
\end{center}

Lemma 3.15 tells us that the fibre of each $B_i$ ($i\in\textbf{A}$) is exactly contained in one $P_j$. Let's consider the map $\mu:\textbf{A}\to \textbf{P}$ such that
\begin{center}
    $i\in\textbf{A}$ $\mapsto$ $j:P_j$ contains the fibre of $B_i$ (along $\pi$).
\end{center}
Without loss of generality, let's assume $1\in\textbf{P}$ and we consider $\mu^{-1}(1)\subset \textbf{A}$. Firstly, it is clear that the fibres $\pi^{-1}(B_i)\cap\pi^{-1}(B_j)=\{0\}$ for any $i,j\in\mu^{-1}(1)$; secondly, the direct sum of Bergman fans $\oplus_{i\in\mu^{-1}(1)}B_i$ can be embedded into $\oplus_{i\in\mu^{-1}(1)}\mathbb{Z}^{d_i}$ where $d_i:=$dim$(B_i)$. Note that 
\begin{center}
    $|\{$rays of $P_1\}|$ $=$ $|P_1|+1$ $=$ $|\{$rays of $\pi_F(P_1)\}|$
\end{center}
by the construction of $\pi_F$ above. Hence we define $\Pi_i:=\pi_F(\pi^{-1}(B_i))\subset\pi_F(P_1)$ for all $i\in\mu^{-1}(1)$. The discussion above implies
\begin{center}
    $|\{$rays of $\Pi_i\}|$ $=$ $|\{$rays of $\pi^{-1}(B_i)\}|$ $=$ $|\{$rays of $B_i\}|$.
\end{center}
We claim that there exists a $\mathbb{Z}$-linear map $\psi_i(\mu^{-1}(1)):\mathbb{Z}^{|P_1|}\to\mathbb{Z}^{d_i}$ such that $\psi_i(\mu^{-1}(1))(\Pi_i)=B_i$. Indeed, let's assume $\Pi_i:=\bigcup_{j=1}^{d_i+1}\mathbb{R}_+(r_j)\subset\pi_F(P_1)$ and we define $g_{l_j}(r_j,r_{d_i+1})$ to be the associated gallery. Hence $\psi_i(\mu^{-1})$ is defined by the following $\mathbb{Z}$ matrix
$$M(\psi_i(\mu^{-1}(1))):=[l_1,l_2,\dots,l_{d_i}]^T.$$
It is obvious to see that $\psi_i(\mu^{-1}(1))(\Pi_i)$ is a Bergman fan over $\mathbb{Z}^{d_i}$ which is isomorphic to $B_i$. Therefore, for $\{i_1,\dots,i_{|\mu^{-1}(1)|}\}=\mu^{-1}(1)$, we define
$$\psi(\mu^{-1}(1)):\psi_{i_1}(\mu^{-1}(1))\times\psi_{i_2}(\mu^{-1}(1))\times\dots\times\psi_{i_{|\mu^{-1}(1)|}}(\mu^{-1}(1)) $$which is a $\mathbb{Z}$-linear map $\mathbb{Z}^{|P_1|}\to\mathbb{Z}^{\sum_{i\in\mu^{-1}(1)}d_i}$. The discussion above gives
$$\psi(\mu^{-1}(1))(P_1)=B_{i_1}\oplus B_{i_2}\oplus\dots\oplus B_{i_{|\mu^{-1}(1)|}}.$$
 Finally, by defining $\psi(\mu^{-1}(j))$ for all $j\in \textbf{P}$, we set
$$\psi:\psi(\mu^{-1}(1))\times\dots\times\psi(\mu^{-1}(|\textbf{P}|))$$
and this $\psi$ is the desired $\mathbb{Z}$-linear map $\mathbb{Z}^{\sum_{i\in\textbf{P}}|P_i|}\to\mathbb{Z}^r$ which has $\psi\circ\pi_F=\pi$ as required.\\

3. Let $M$ be a TR function such that $M\cdot F=1$. Then Theorem 3.11 allows us to assume $M:=\max(0,h_1,\ldots,h_m)$ such that $g_{h_i}(v_{m+1},v_i)\in\mathrm{Gly}(F)$ after re-arranging the indices. Note that $$M=\max(0,(1,0,\ldots,0),(0,1,\ldots,0),\ldots,(0,\ldots,1))\circ[h_1,\ldots,h_m]^T.$$
On the one hand, let \textbf{P} be a canonical partition containing an index $m+1\in\textbf{P}$ where $v_{m+1}$ is the generator of $F$ satisfying $M\le\mathscr{M}_{max}(v_{m+1})$. Then the matrix $[h_1,\ldots,h_m]^T$ is a sub-matrix of the matrix defined by $\pi_F$ as we constructed above.

On the other hand, the TR function $\max(0,(1,0,\ldots,0),\ldots,(0,\ldots,1))$ can be naturally extended to a TR function over $\mathbb{R}^{\sum_{i\in\textbf{P}}|P_i|}$ which is a regular sequence of ${\pi_F}_*(F)$.

As a consequence, we conclude that $M$ is lifted from a non-negative TR function extended from $\max(0,(1,0,\ldots,0),\ldots,(0,\ldots,1))$ along $\pi_F$ as required.
\end{proof}

We close this section with the following curious consequence, which
sheds further light on the combinatorial assembly of \(F\).
\begin{prop}
    Let $m\in\mathbf{P}$. If $F$ is irreducible and $F\neq P_j$ for all $j\in\mathbf{P}$, then dim$(\langle P_m\rangle)$ $=|P_m|+1$.  
\end{prop}

\begin{proof}
    Set $P_m=\{g_{l_i}(v_i,v_m)\mid1\le i\le k\}$. If $\{v_1,v_2,\dots,v_k\}$ is linearly dependent, say
    $$\sum_{i=1}^k\lambda_iv_i=0$$
    with an index $\alpha:\lambda_\alpha\neq 0$, then $l_\alpha(\sum_{i=1}^k\lambda_iv_i)=0$ implies that there exists a generator, say $v_\beta$ such that $1\le \beta\le k$, $l_\alpha(v_\beta)\neq 0$. This contradicts the definition of $g_{l_\alpha}(v_\alpha,v_m)$. Therefore $\{v_1,v_2,\dots,v_k\}$ is linearly independent. So dim$(\langle P_m\rangle)\ge |P_m|=k$. If $\{v_m,v_1,v_2,\dots,v_k\}$ is linearly independent, then dim$(\langle P_i\rangle)=|P_i|+1$ as required. If the set is linearly dependent, say
    $$\sum_{i=1}^k\lambda_i'v_i+\lambda_m'v_m=0,$$
    then for any $1\le t\le k$,
    $$l_t(\sum_{i=1}^k\lambda_i'v_i+\lambda_m'v_m)=l_t(\lambda_t'v_t)+l_t(\lambda_m'v_m)=\lambda_t'-\lambda'_m.$$
    So
    $$\lambda'_i=\lambda'_j=\lambda'_m\ \ \ \forall 1\le i\neq j\le k.$$
    This implies that $v_m+\sum_{i=1}^kv_i=0$. However, this is exactly the balancing condition on the sub-fan $F_K$ defined by the rays generated by $\{v_m,v_1,v_2,\dots,v_k\}$: $$\sum_{i=1}^k\omega_iv_i+\omega_mv_m=0=\sum_{i=1}^kv_i+v_m.$$
    (Note that $\omega_i=1$ comes from the definition of gallery).

    Hence $F_k$ is a $1$-dimensional tropical sub-fan of $F$. Therefore $F=F_k+(F\setminus F_K)$ which contradicts the irreducibility since $F\setminus F_k$ is non-empty by assumption.\\
\end{proof}

\section{Finiteness theorem for strongly regular $2$-cycles}

One of our interests is to learn the 2-dimensional minimal model programme for 2-cycles as we mentioned in the Introduction. This section plays an essential role as a toolkit for this study. In dimension 2, the Hodge--Lefschetz package reduces to the Hodge index theorem for 2-dimensional tropical fans, as in \cite{BabaeeHuh19}. As an analogue of the Hodge index hypothesis, we restrict our $2$-cycles $[F]$ to satisfy the following: Let $M_1,M_2$ be two TR functions over $\mathbb{R}^n$. Then
$$M_1\cdot M_1\cdot [F]\le 1;\ M_1\cdot M_2\cdot [F]\le 1;\ M_2\cdot M_2\cdot [F]\le 1.$$
Although this condition is neither weaker nor stronger than the Hodge index theorem—$[F]$ is even not necessarily regular—it covers all the cases we are interested in: we say $[F]$ satisfies the Hodge index theorem if $(M_1\cdot M_2\cdot [F])^2\ge (M_1\cdot M_1\cdot [F])(M_2\cdot M_2\cdot [F])$. We classify the above inequalities into the following three cases:

1. If $[F]$ satisfies the Hodge index theorem, and $(M_2\cdot M_2\cdot [F])(M_1\cdot M_1\cdot [F])>0$, then $[F]$ is regular over the regular sequence $M_1,M_2$ which means $[F]$ satisfies our inequalities assumption above.

2. If $[F]$ satisfies the Hodge index theorem which has $M_1\cdot M_2\cdot [F]=1$, but $M_1\cdot M_1\cdot [F]=0$, then we may check whether $[F]$ defines a regular fibration when $M_2\cdot M_2\cdot [F]\le 1$ (Definition 2.15). Furthermore, can $[F]$ define a regular fibration for arbitrary $M_2\cdot M_2\cdot [F]$? As the first motivation for putting forward this view, we give the following example:

\begin{exa}
    Let $F$ be a 2-dimensional tropical fan over $\mathbb{R}^3$ given by the following graph:
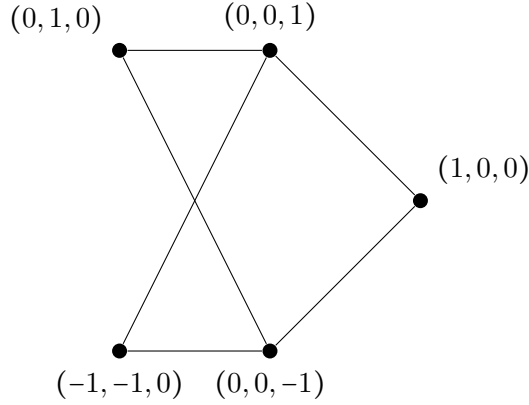
\begin{center}
\begin{tikzpicture}[scale=2,
        vertex/.style={circle,fill=black,inner sep=2pt},
        every label/.style={font=\small}]

\node[vertex,label=above right:{$(1,0,0)$}] (a) at ( 1, 0) {};
\node[vertex,label=above left:{$(0,1,0)$}]  (b) at ( -1,1) {};
\node[vertex,label=below:{$(-1,-1,0)$}]       (c) at (-1,-1) {};
\node[vertex,label=above:{$(0,0,1)$}]       (d) at (0,1) {};
\node[vertex,label=below:{$(0,0,-1)$}]       (e) at (0,-1) {};

\draw (a) -- (d);
\draw (a) -- (e);
\draw (b) -- (d);
\draw (b) -- (e);
\draw (c) -- (d);
\draw (c) -- (e);

\end{tikzpicture}
\end{center} 
where the vertices are the rays of $F$ and the edges are the cones generated by the rays given by the endpoints. Let $M_1,M_2$ be two non-negative TR functions over $\mathbb{R}^3$ such that $M_1=\max(0,(1,0,0)^\vee)$ and $M_2=\max(0,(0,0,1)^\vee)$. Then one can calculate that 
$$M_1\cdot [F]=[(\{\mathbb{R}_+\cdot(0,0,1),\mathbb{R}_+\cdot(0,0,-1),\textbf{0}\},\omega=1)],$$
$$M_2\cdot[F]=[(\{\mathbb{R}_+\cdot(1,0,0),\mathbb{R}_+\cdot(0,1,0),\mathbb{R}_+\cdot(-1,-1,0),\textbf{0}\}),\omega=1].$$
Therefore, we have $M_1\cdot M_2\cdot [F]=1=M_2\cdot M_2\cdot [F]$ but $M_1\cdot M_1\cdot [F]=0$. Define $\pi:\mathbb{R}^3\to\mathbb{R}^3/(0,0,1)$. Then $\pi(|F|)=|M_2\cdot[F]|$, and the central fibre $\pi|_F=|M_1\cdot[F]|$. As a consequence, both of them have structures of regular cycles. So $\pi$ is a regular fibration of $F$ as required.
\end{exa}

3. If $M_1\cdot M_2\cdot [F]=0$ but the other two are $1$, then $[F]$ does not satisfy the Hodge index theorem. So we get some interesting examples which contradict with the Hodge index theorem as by-products.

\begin{defin}[Strongly regular 2-cycle]
A $2$-cycle $[F]$ over $\mathbb{R}^n$ is called strongly regular over $M_1,M_2$ if there exist non-negative TR functions $M_1,M_2$ such that
$$M_1\cdot M_2\cdot [F]=1;\ M_1\cdot M_1\cdot [F]\le 1;\ M_2\cdot M_2\cdot[F]\le1.$$
\end{defin}

\begin{convention}
    In this section, we will always assume $T_1,T_2$ to be two TR functions over $\mathbb{R}^n$ such that
$$T_1:=\max(v_0,v_1,v_2,\dots,v_k),$$
$$T_2:=\max(w_0,w_1,w_2,\dots,w_{n-k}),$$
and $v_1,v_2,\dotsc,v_k,w_1,\dotsc,w_{n-k}$ are linearly independent over $(\mathbb{Z}^n)^\vee$ and $v_0=w_0=0$ for convenience.\\
\end{convention}

\subsection{Settings and examples}

The following theorem provides the section’s first finiteness result and plays a pivotal role in the proof of our main theorem.

\begin{theor}
    Given two non-negative TR functions $T_1,T_2$ that satisfy Convention 4.3, there are finitely many $2$-cycles $[L]$ whose support $|L|$ is a 2-plane such that $T_1\cdot T_2\cdot [L]=1$ and $T_1\cdot T_1\cdot [L]\le 1$, $T_2\cdot T_2\cdot [L]\le 1$.
\end{theor}

\begin{proof} From $T_1\cdot T_2\cdot[L]=1$ we conclude that the weight function of any representative of $[L]$ can only be the constant $1$. Hence it remains to show that we have only finitely many possible $|L|$.

    Now we have two cases: $T_1\cdot [L]$, $T_2\cdot [L]$ are two lines ($T_1\cdot T_1\cdot [L]=0=T_2\cdot T_2\cdot [L]$); or one of them is not a line (say $T_1\cdot T_1\cdot [L]=1$).\\
    
    \textit{case 1:} For each $i=1,2$. Assume $T_i\cdot [L]=\langle\rho_i\rangle$ to be a line with primitive generator $\rho_i$ (clearly $\rho_1\neq \rho_2$). Then $[V(T_i)\cap_{st}L]=\langle\rho_i\rangle$. Thus $\langle\rho_i\rangle$ lies entirely in the lineality spaces of $V(T_i)$, called $W_i$. We denote them by $W_1:=\bigcap_{i=1}^k(v_i)^\perp$ and $W_2:=\bigcap_{j=1}^{n-k}(w_j)^\perp$. Moreover, $T_2\cdot\langle\rho_1\rangle=1=T_1\cdot\langle\rho_2\rangle$ implies that there are two index sets $A\subset\{1,\dots,k\}$, $B\subset\{1,\dots,n-k\}$ such that 
    $$v_i(\rho_2)=1\ \forall\ i\in A\ \ v_j(\rho_2)=0\ \forall\ j\notin A;$$
    $$w_i(\rho_1)=1\ \forall\ i\in B\ \ w_j(\rho_1)=0\ \forall\ j\notin B.$$
    This is because $1=T_2\cdot\langle\rho_1\rangle=\omega(\rho_1)T_2(\rho_1)+\omega(-\rho_1)T_2(-\rho_1)$. Therefore $\omega(\rho_1)=1=\omega(-\rho_1)$,  so we may assume $T_2(\rho_1)=1$ and thus $T_2(-\rho_1)=0$. $T_2(\rho_1)=1$ implies that there exists an index set $B\subset\{1,\dots,n-k\}$ with $w_i(\rho_1)=1$ for all $i\in B$, and $w_j(\rho_1)<w_i(\rho_1)$ for all $j\notin B$. Let $j\notin B$. If $w_j(\rho_1)<0$. Then $w_j(-\rho_1)>0$ which contradicts $T_2(-\rho_1)=0$. Hence $w_j(\rho_1)=0$ for all $j\notin B$. The analogous statement holds for $T_1$.
    Therefore, let $M$ be the integer matrix $[v_1,\dots,v_k,w_1,\dots,w_{n-k}]^T$, and thus $\rho_2$ satisfies:
    $$ M(\rho_2)=(\textbf{x}_A,0,\dots,0)^T\in\mathbb{Z}^n\ \ A\in 2^{\{1,2,\dots,k\}}\setminus\emptyset$$
    where $\textbf{x}_A$ is a vector in $\mathbb{Z}^k$ such that
    $$\begin{cases}
        e_i(\textbf{x}_A)=1,&i\in A\\
        e_j(\textbf{x}_A)=0,&j\notin A
    \end{cases}$$
    where $e_i,e_j$ are the standard dual basis elements of $(\mathbb{Z}^k)^\vee$. Hence there are finitely many such $\rho_2$ since $M$ is invertible. We can solve for all possible $\rho_1$ by the same way. Therefore, there are finitely many such 2-planes $|L|:=\langle\rho_1,\rho_2\rangle$.\\

    \textit{case 2:} If $T_1\cdot [L]$ is not a line, then it is a tropical curve on the 2-plane $|L|$. The hypothesis tells us $T_1\cdot T_1\cdot [L]\le 1$ and $T_1\cdot T_1\cdot [L]=[(T_1\cdot [L])\cap_{st}(T_1\cdot [L])]>0$. So $T_1\cdot T_1\cdot [L]=1=T_2\cdot T_1\cdot [L]$. From the full classification of 1-D tropical fans in Section 3, $T_1\cdot [L]$ is a $1$-D Bergman fan over $|L|\cong\mathbb{R}^2$. Therefore, let's assume $C:=(\bigcup_{i=1}^3r_i,\omega_i=1)$ to be the corresponding 1-D tropical curve $T_1\cdot [L]$.
    
     $T_1\cdot C=1=T_2\cdot C$ gives
    $$\sum_{i=1}^3T_1(r_i)=1=\sum_{i=1}^3T_2(r_i).$$
    
    Without loss of generality, let's assume $T_1(r_1)=1$ and $T_1(r_2)=T_1(r_3)=0$. We denote $\Delta_k\subseteq\{1,2,\dots,k\}$ such that $i\in\Delta_k$ if and only if $v_i(r_1)=1$.
    
    The following two bullets record all possibilities of $v_i(r_j)$ for $i\in\{1,\dots,k\}$ and $j=1,2,3$. We will see that there are only finitely many choices for $v_i(r_j)$.\\

    $\bullet$ for any $i\notin\Delta_k$: $v_i(r_1)=v_i(r_2)=v_i(r_3)=0$.

    Clearly $v_i(r_1)\le 0$ for all $i\notin\Delta_k$. If $v_i(r_1)<0$, then we apply $v_i$ to the balancing condition $\sum_{t=1}^3r_i=0$ and obtain an index $j\in\{2,3\}$ with $v_i(r_j)>0$. However, this contradicts $T_1(r_2)=0=T_1(r_3)$, so $v_i(r_1)=0$. Furthermore, it is clear that $v_i(r_2)\le 0$ and $v_i(r_3)\le 0$. Then balancing condition implies
    $$v_i(\sum_{j=1}^3r_i)=0=v_i(r_2)+v_i(r_3). $$
    Therefore we can only have $v_i(r_2)=0=v_i(r_3)$ for all $i\notin\Delta_k$.\\

    $\bullet$ $\Delta_k=\Delta(2)\dot \cup\Delta(3)$ where $i\in\Delta(2)$ if and only if $v_i(r_2)=-1$ and $v_i(r_3)=0$. Similarly, $j\in\Delta(3)$ if and only if $v_j(r_3)=-1$, $v_j(r_2)=0$.

    This can be obtained directly from the balancing condition $\sum_{j=1}^3r_j=0$.\\
    
    Indeed, the two bullets above reveal all possibilities of $v_i(r_j)$ for all $i\in\{0,1,\dots,k\}$ and $j\in\{1,2,3\}$. By defining an analogous set $\Delta_{n-k}\subset\{1,2,\dots,n-k\}$ we can play the same game among all linear functional $w_i$ of $T_2$, and we can obtain all possibilities of $w_i(r_j)$ for all $i\in\{0,1,\dots,n-k\}$ and $j\in\{1,2,3\}$ as well. Finally, $r_1,r_2,r_3$ can be solved by the image set of invertible matrix $[v_1,\dots,v_k,w_1,\dots,w_{n-k}]^T$ as usual. Therefore, the 2-plane $|L|=\langle r_1,r_2,r_3\rangle$ has only finitely many possibilities in this case.
\end{proof}

The following two counter-examples explain why we cannot drop any one of the restrictions in Theorem 4.4:

\begin{exa}[Counter example 1] We cannot drop the condition $T_1
\cdot T_1\cdot [L]\le 1$ or $T_2\cdot T_2\cdot [L]\le 1$; here is a counter example. Let $T_1:=\max(0,v_1,v_2)$ and $T_2:=\max(0,v_3,v_4)$ where $\{v_1,\dots,v_4\}$ forms the standard basis of $\mathbb{Z}^4$. Let $\rho=(1,1,0,0)\in\mathbb{Z}^4$, $\tau_{ab}=(a,b,0,-1)\in\mathbb{Z}^4$ for arbitrary $a,b\in\mathbb{Z}$. We set the 2-plane $L_{ab}:=(\langle\rho,\tau_{ab}\rangle,\omega=1)$.

Indeed, $T_2\cdot [L_{ab}]=[(\langle\rho\rangle,\omega=1)]$. Note that $[V(T_2)\cap_{st}L_{ab}]=[\langle\rho\rangle]$. So $T_2\cdot [L_{ab}]$ is a line whose support is $\langle\rho\rangle$ with a positive weight. It remains to show that $\omega(T_2\cdot [L_{ab}])=1$. Note that $(1,0,0,0),\rho,(0,0,1,0),\tau_{ab}$ form a $\mathbb{Z}$-basis, so $[\mathbb{Z}^4\cap L_{ab}:\mathbb{Z}^4_{\rho}+\mathbb{Z}^4_{\tau_{ab}}]=1$ and this means that $\pm\tau_{ab}$ are exactly the primitive generators pointing from $\langle\rho\rangle$. Hence
$$\omega(T_2\cdot [L_{ab}])=\omega(\rho)= T_2(\tau_{ab})-T_2(\rho)+T_2(-\tau_{ab})-T_2(\rho)=0-0+v_4(-\tau_{ab})-0=1.$$
So we have $T_2\cdot [L_{ab}]=[(\langle\rho\rangle,\omega=1)]$. It is clear that $T_1\cdot[(\langle\rho\rangle,\omega=1)]=1$ but $T_2\cdot[(\langle\rho\rangle,\omega=1)]=0$.
    This implies that any 2-plane of the form $L_{ab}$ satisfying $T_1\cdot T_2\cdot [L_{ab}]=1$ and $T_2\cdot T_2\cdot [L_{ab}]=0$. Hence the finiteness fails if we drop the restriction of $T_1\cdot T_1\cdot [L]$.
\end{exa}

\begin{exa}[Counter example 2] We cannot drop the condition $T_1\cdot T_2\cdot [L]=1$ either. In other words, there are infinitely many 2-planes $L$ with $T_1\cdot T_1\cdot [L]=0=T_2\cdot T_2\cdot [L]$. We also assume $T_1:=\max(0,v_1,v_2)$ and $T_2:=\max(0,v_3,v_4)$ where $\{v_1,\dots,v_4\}$ forms the standard basis of $\mathbb{Z}^4$. Define the 2-plane $L_{abcd}:=\langle(a,b,0,0),(0,0,c,d)\rangle$ with $\mathrm{gcd}(a,b)=1=\mathrm{gcd}(c,d)$. Clearly, $T_1\cdot[L_{abcd}]=[\langle(0,0,c,d)\rangle]$ and $T_2\cdot[L_{abcd}]=[\langle(a,b,0,0)\rangle]$. Thus $T_1\cdot T_1\cdot [L_{abcd}]=0=T_2\cdot T_2\cdot [L_{abcd}]$. Therefore we have infinitely many such $L_{abcd}$ so the finiteness fails.\\
    
\end{exa}

The next lemma will be used in the main proof of this section which is an analogue of Theorem 4.4. It reveals more finiteness results based on the idea of the proof in Theorem 4.4.

\begin{lemma}
    Let $T_1,T_2$ be two non-negative TR functions in Convention 4.3. Then
    
    a. there are finitely many $2$-cycles $[L]$ whose support is a 2-plane satisfying $T_1\cdot T_2\cdot [L]=0=T_1\cdot T_1\cdot [L]$ while $T_2\cdot T_2\cdot [L]=1$, 

    b. there does not exist any $2$-cycle $[L]$ whose support is a 2-plane satisfying $T_1\cdot T_2\cdot[L]=0=T_1\cdot T_1\cdot [L]=T_2\cdot T_2\cdot [L]$, and

    c. there does not exist any $2$-cycle $[L]$ whose support is a 2-plane satisfying $T_1\cdot T_2\cdot[L]=0$ but $T_1\cdot T_1\cdot [L]=1=T_2\cdot T_2\cdot [L]$.

\end{lemma}

\begin{proof}
   a. According to the proof of case 2 in Theorem 4.4, the 1-cycle $T_2\cdot [L]$ is a 1-D Bergman fan $C:=(\bigcup_{i=1}^3r_i,\omega_i=1)$. To convince the reader, we introduce another way to see it: $T_2\cdot T_2\cdot [L]=1$ gives $$(T_2\cdot[L])\cap_{st}^L(T_2\cdot[L])=1$$where $\cap_{st}^L$ means the stable intersection only defined in $|L|$. Therefore, a 1-dimensional tropical fan whose self-intersection number is $1$ if and only if it is a 1-D Bergman fan equipped with constant weight $1$.
   
   On the one hand, $T_1\cdot T_2\cdot [L]=0=T_1\cdot [C]$ gives $v_i(r_j)=0$ for all $1\le i\le k$ and $j\in\{1,2,3\}$. On the other hand, $T_2\cdot T_2\cdot [L]=1=T_2\cdot [C]$ allows us to define the set $\Delta_{n-k}\subset\{1,2,\ldots,n-k\}$ such that $i\in\Delta_{n-k}$ if and only if $w_i(r_1)=1$ after assuming $T_2(r_1)=1$, $T_2(r_2)=0=T_2(r_3)$. The process in the proof of case 2 in Theorem 4.4 reveals finitely many possible $w_i(r_j)$ for all $1\le i\le n-k$ and $j\in\{1,2,3\}$. Combine these, the image set of each $r_i$ under the invertible matrix $[v_1,\ldots,v_k,w_1,\ldots,w_{n-k}]^T$ is finite. Therefore, the solution set is also finite. Hence, we only obtain finitely many supports $|L|=\langle r_1,r_2,r_3\rangle$ as required.\\

   b. Assume there exists a $2$-cycle $[L]$ satisfying the assumption. If $T_1\cdot [L]=0=T_2\cdot [L]$ be two 1-cycles whose weights are $0$, then $|L|\subseteq W_1$, $|L|\subseteq W_2$ where $W_1$ (resp. $W_2$) is the lineality space of $T_1$ (resp. $T_2$). However, this gives $|L|\subseteq W_1\cap W_2=\{0\}$ which is impossible. Hence we may assume $T_1\cdot [L]\neq 0$. Then $|T_1\cdot [L]|\subset W_1$. Moreover, $T_1\cdot T_2\cdot [L]=0$ implies that $T_1\cdot [L]$ lives in $W_2$. This contradicts $W_1\cap W_2=\{0\}$ again.\\

   c.  Theorem 4.15 of \cite{Esterov25} gives that $[L]$ satisfies the Hodge index theorem. So we cannot have the case of $$(T_1\cdot T_2\cdot [L])^2=0<1=(T_1\cdot T_1\cdot [L])(T_2\cdot T_2\cdot [L]).$$
\end{proof}

We give the following example with an explicit calculation to help readers understanding the computation in the proof of Theorem 4.4.

\begin{exa} 
    Let $M_1=\max(0,v_1,v_2)$, $M_2=\max(0,v_3,v_4)$ be two TR functions over $\mathbb{R}^4$ such that det$(v_1,v_2,v_3,v_4)>0$. We represent the hypersurface of $M_i$ $(i=1,2)$ $$V(M_i)=\bigcup_{j=1}^3\widetilde{H}^i_j$$
as the union of three half-hypersurfaces $\widetilde{H}^i_j$. In particular, we can precisely express the hypersurfaces of $M_1$ (and similar for $M_2$):

$\widetilde{H}^1_1:=v_1^\perp\cap\{x\in\mathbb{R}^4|v_2(x)\le 0\}$;

$\widetilde{H}^1_2:=v_2^\perp\cap\{x\in\mathbb{R}^4|v_1(x)\le0\}$;

$\widetilde{H}^1_3:=\{x\in\mathbb{R}^4|v_1(x)=v_2(x)\ge 0\}$.

We denote the hyperplane $H^i_j:=\langle\widetilde{H}^i_j\rangle$, and set $L_{ij}:=H^1_i\cap H^2_j$ to be a 2-plane. We call the lineality space of $V(M_i)$ by $W_i=:\bigcap_{j=1}^3H^i_j$. Notice that $W_1\cap W_2=\{0\}$.\\

  Now  Let $L$ be a rational 2-plane with $M_1\cdot M_2\cdot [L]=1$. If $M_1\cdot M_1\cdot [L]=M_2\cdot M_2\cdot [L ]=0$, then $L=L_{ij}$ for some $i,j\in\{1,2,3\}$ with weight $1$.\\

To see this, let's firstly assume $L\neq W_1,W_2$.

    We claim that $M_1\cdot [L],M_2\cdot [L]$ are two lines. Note that $M_i\cdot [L]=[V(M_i)\cap_{st}L]$, so it is either a line that lies in the lineality space $W_i$, or a 1-D tropical fan with three rays. The stable intersection between $L$ and $V(M_i)$ has three rays which implies that:
    \begin{center}
        $[M_i\cap_{st}L]=[\bigcup_{j=1}^3$ the intersection between $L$ and $\widetilde{H}^i_j$].
    \end{center}
    Now let's consider $M_1\cdot [L]$. Assume $M_1\cdot [L]$ to be a 1-cycle with three rays whose primitive generators are $r_1,r_2,r_3$, then the above statement allows us to set
    $$r_i\subset\widetilde{H}^1_i\ \ i=1,2,3.$$

   This gives
    $$r_1\subset\widetilde{H}^1_1\Longrightarrow v_1(r_1)=0,v_2(r_1)< 0,$$

    $$r_2\subset\widetilde{H}^1_2\Longrightarrow v_2(r_2)=0,v_1(r_2)<0,$$

    $$r_3\subset \widetilde{H}^1_3\Longrightarrow v_1(r_3)=v_1(r_3)>0$$
    since $r_1,r_2,r_3\not\subset W_1$. Therefore, we can calculate $M_1\cdot M_1\cdot [L]$ by
    $$\sum_{i=1}^3M_1(r_i)=0+0+v_1(r_3)=0+0+v_2(r_3)>0$$
    which contradicts $M_1\cdot M_1\cdot [L]=0$. Therefore, $M_1\cdot [L]$ (and $M_2\cdot [L]$) are two lines live in $W_1$ (and $W_2$) respectively. Now we set $M_1\cdot L=\langle\rho_1\rangle$ and $M_2\cdot L=\langle\rho_2\rangle$ where $\rho_1,\rho_2$ are two rays.

    On the one hand, $\langle\rho_i\rangle\subset W_i$ for $i=1,2$ as we discussed above,

    On the other hand, $M_1\cdot M_2\cdot [L]=1$ implies that
    $$1=M_2\cdot (M_1\cdot [L])=M_2\cdot (\langle\rho_1\rangle)=\omega(\rho_1)M_2(\rho_1)+\omega(\rho_1)M_2(-\rho_1).$$
    The non-negativity of $\omega(\rho_1),M_2(\rho_1),M_2(-\rho_1)$ tell us that we may assume
    $$\omega(\rho_1)=1,M_2(\rho_1)=1,M_2(-\rho_1)=0.$$
    Finally, $M_2(\rho_1)=1, M_2(-\rho_1)=0$ means that: if we have
    $$v_3(\rho_1)=1,v_4(\rho_1)<0.$$
    Then $M_2(-\rho_1)=v_4(-\rho_1)>0$, contradiction. So we either have
    $$v_3(\rho_1)=1,v_4(\rho_1)=0,$$
    or
    $$v_3(\rho_1)=0,v_4(\rho_1)=1,$$
    or
    $$v_3(\rho_1)=v_4(\rho_1)=1.$$
    Therefore, $\rho_1$ lives in one of ${H}^2_i$ for $i\in\{1,2,3\}$.

    Combine these, we know that $\rho_1\subset W_1,$ $\rho_1\subset{H}^2_j$, hence $\rho_1\subset W_1\cap{H}^2_j$. In the similar way, we can obtain that $\rho_2=W_2\cap H^1_i$ for some $i\in\{1,2,3\}$. Hence $L=H^1_i\cap H^2_j=L_{ij}$ as required.

    In the case of $L=W_1,W_2$, let's say $L=W_1$. Then $M_2\cdot M_2\cdot [L]=$det$(W_1,W_2)>0$, contradiction.\\
\end{exa}

\subsection{2-dimensional galleries} 

Galleries play an essential role in the study of the conjecture of ‘minimal model programme' (Question 1.2). We state A fundamental result for the gallery of 2-dimensional tropical fans in this sub-section which will help readers learn the structure of 2-dimensional galleries clearly.

\begin{defin}
    Let \(\rho\subset F\) be a ray (1–dimensional cell). We define
$$ad(\rho):=\{\sigma\in F^{(2)}|\rho\subset\sigma\}$$
to be the collection of maximal faces of \(F\) containing \(\rho\).
\end{defin}

The next proposition clarifies the internal structure of any
2–dimensional gallery in \(F\): every such gallery forms a cycle inside \(F\).

\begin{prop}
Let $F$ be a 2-dimensional tropical fan with a non-empty gallery $g(F,L_1,L_2)$. Then for any ray $\rho$ in the support of $|g(F,L_1,L_2)|$, we have $|ad(\rho)\cap g(F,L_1,L_2)|=2$.
\end{prop}

\begin{proof}

We denote $\pi:\mathbb{R}^n\to\mathbb{R}^n/(V(L_1)\cap V(L_2))$ to be a projection map.
   
    Now let $\rho$ be a ray of $ g(F,L_1,L_2)$, we have $|ad(\rho)\cap g(F,L_1,L_2)|\le 2$. Otherwise there must exist some maximal faces in $\pi(F)$ with weight larger than $1$ which is impossible. 
    
    Clearly there exists at least one $\sigma\in ad(\rho)\cap g(F,L_1,L_2)$, so $1\le |ad(\rho)\cap g(F,L_1,L_2)|\le 2$. It remains to show that $|ad(\rho)\cap g(F,L_1,L_2)|=2$ for all rays $\rho$. By contradiction, if $|ad(\rho)\cap g(F,L_1,L_2)|=1$, i.e. $ad(\rho)\cap g(F,L_1,L_2)=\sigma$. Then balancing condition tells
    $$\sum_{\tau_i\in ad(\rho)\setminus \sigma}\omega(\tau_i)v_{\tau_i/\rho}+\omega(\sigma)v_{\sigma/\rho}\subset \langle\rho\rangle$$
    where $ad(\rho):=\{\tau_1,...,\tau_{|ad(\rho)|-1},\sigma$\}.
    
    For each $\tau_i\in ad(\rho)\setminus\sigma$,  $\langle{\tau_i}\rangle\cap ker(\pi)\neq 0$ and $\rho\not\subset ker(\pi)$ by the construction of $g(F,L_1,L_2)$, hence $\langle{\tau_i}\rangle\not\subseteq ker(\pi)$, so the intersection $\langle{\tau_i}\rangle\cap ker(\pi)$ must be a 1-D subspace, let's say this 1-D space is $\langle n_{\tau_i}\rangle$. Now we have
    $$\langle{\tau_i}\rangle\cap ker(\pi)=\langle n_{\tau_i}\rangle.$$ In this case, there exists $u_i\subset \langle\rho\rangle$ such that
    $$u_i+m_iv_{\tau_i/\rho}=n_{\tau_i}$$for some non-zero  $m_i\in \mathbb{Q}$ and $v_{\tau_i/\rho}$ is the primitive generator from $\langle\rho\rangle$ to $\tau_i$. The balancing condition implies that

    $$(\sum_{i=1}^{|ad(\rho)|-1}\omega(\tau_i)v_{\tau_i/\rho}+1\cdot v_{\sigma/\rho})\subset \langle\rho\rangle.$$
    Therefore, we have
    $$(\prod_{i=1}^{|ad(\rho)|-1}m_i)\cdot(\sum_{i=1}^{|ad(\rho)|-1}\omega(\tau_i)v_{\tau_i/\rho}+1\cdot v_{\sigma/\rho}+\sum_{i=1}^{|ad(\rho)|-1}\frac{\omega(\tau_i)u_{i}}{m_i})\subset \langle\rho\rangle.$$
    So
    $$\sum_{i=1}^{|ad(\rho)|-1}(\omega(\tau_i)\prod_{j:j\neq i}m_j(u_{i}+m_iv_{\tau_i/\rho}))+(\prod_{i=1}^{|ad(\rho)|-1}m_i)v_{\sigma/\rho}\subset \langle\rho\rangle.$$
    So
    $$\sum_{\tau_i}(\omega(\tau_i)\prod_{j:j\neq i}m_j(n_{\tau_i}))+(\prod_{i=1}^{|ad(\rho)|-1}m_i)v_{\sigma/\rho}\subset \langle\rho\rangle.$$
    After applying $\pi$ to both sides we have
    $$(\prod_{i=1}^{|ad(\rho)|-1}m_i)\pi(v_{\sigma/\rho})\subset \pi(\langle\rho\rangle).$$
    However, $v_{\sigma/\rho}$ points from $\rho$ to the interior of $\sigma$, so this tells us that $\langle\sigma\rangle\cap ker(\pi)\neq 0$, contradiction!

    Therefore $|ad(\rho)\cap g(F,L_1,L_2)|=2$ as required.
\end{proof}

\subsection{The finiteness theorem} Theorem 4.13—our paper’s keystone result—appears immediately after a monotonicity lemma concerning the faces of \(F\).

\begin{lemma}
 Let $F$ be a 2-dimensional tropical fan over $\mathbb{R}^n$. Then $$T_1\cdot T_2\cdot [F]\ge T_1\cdot T_2\cdot [\langle\sigma\rangle]$$ for any face $\sigma\in F^{(2)}$.
\end{lemma}

\begin{proof}
If $T_1\cdot T_2\cdot [F]=0$, then $F$ lives in the lineality space of $V(T_1)\cap_{st} V(T_2)$ so $T_1\cdot T_2\cdot[\langle\sigma\rangle]$ follows immediately.

    If $T_1\cdot T_2\cdot [F]>0$, then intersection product can be viewed as stable intersection by Theorem 2.12:
    $$T_1\cdot T_2\cdot [F]=[V(T_1)\cap_{st}(V(T_2)\cap_{st}F)]=[(V(T_1)\cap_{st}V(T_2))\cap_{st}F]$$
    which is a 0-cycle with some positive weight $(\textbf{0},\omega_{\textbf{0}})$. In other words, the stable intersection between $F$ and $V(T_1)\cap_{st}V(T_2)$ is the origin point $\textbf{0}$, and each face contributes a non-negative multiplicity to the weight of $\textbf{0}$. For any $\sigma\in F$, the contribution of $\sigma$ to $T_1\cdot T_2\cdot [F]$ is equal to the stable intersection number
    $$ (V(T_1)\cap_{st}V(T_2))\cap_{st}(\sigma+\textbf{v}) $$
    where $\textbf{v}$ is a generic vector such that the shifting under $\textbf{v}$ forces $\textbf{0}$ living in the interior of $\sigma$. This intersection is exactly the point $\textbf{0}$ implies that
    $$(V(T_1)\cap_{st}V(T_2))\cap_{st}(\sigma+\textbf{v})=(V(T_1)\cap_{st}V(T_2))\cap_{st}\langle\sigma\rangle=T_1\cdot T_2\cdot[\langle\sigma\rangle] $$
    and finally the inequality follows from $T_1\cdot T_2\cdot [F]\ge(V(T_1)\cap_{st}V(T_2))\cap_{st}(\sigma+\textbf{v})$.
\end{proof}

\begin{theor}
   Let $T_1,T_2$ be two non-negative TR functions over $\mathbb{R}^n$ such that
$$T_1:=\max(v_0,v_1,v_2,\dots,v_k),$$
$$T_2:=\max(w_0,w_1,w_2,\dots,w_{n-k}),$$
 $v_1,v_2,\dotsc,v_k,w_1,\dotsc,w_{n-k}$ are linearly independent over $(\mathbb{R}^n)^\vee$ and $v_0=w_0=0$. Then there are finitely many $2$-cycles $[F]$ over $\mathbb{Z}^n$ such that $T_1\cdot T_2\cdot [F]\le1$, $T_1\cdot T_1\cdot [F]\le 1$, and $T_2\cdot T_2\cdot [F]\le1$.
\end{theor}

\begin{proof}

    We denote $L^1(i,j)\le T_1$ to be the tropical binomials of $T_1$ defined by $\max(v_i,v_j)$ and $L^2(a,b)\le T_2$ to be $\max(w_a,w_b)$. The monotonicity of Tropical intersection product yields the following:
    
    $$\begin{cases}L^1(i,j)\cdot L^2(a,b)\cdot [F]\le 1;&0\le i\neq j\le k,\ 0\le a\neq b\le n-k\\
    L^1(i,j)\cdot L^1(i',j')\cdot [F]\le 1;& 0\le i,j,i',j'\le k,\ |\{i,j\}\cap\{i',j'\}|\le 1\\
    L^2(a,b)\cdot L^2(a',b')\cdot [F]\le 1;& 0\le a,b,a',b'\le n-k
,\ |\{a,b\}\cap\{a',b'\}|\le 1.    \end{cases}$$
For any such $2$-cycles $[F]$, we refer to $\mathrm{G}\subset\mathrm{Gly}([F])$ as the set of galleries defined by the tropical binomials $\{L^1(i,j),L^2(a,b)\}$, $\{L^1(i,j),L^1(i',j')\}$ and $\{L^2(a,b),L^2(a',b')\}$ above.
  The inequalities above conclude that: for any $\sigma\in F^{(2)}$, either $\sigma$ lives in some galleries from $\mathrm{G}$, or it does not live in any gallery from $\mathrm{G}$. Assume there exists a $\sigma\in F^{(2)}$ which does not live in any gallery in $\mathrm{G}$. Then we claim that
   $$T_1\cdot T_2\cdot [\langle\sigma\rangle]=0=T_1\cdot T_1\cdot[\langle\sigma\rangle]=T_2\cdot T_2\cdot [\langle\sigma\rangle].$$
   To see this, assume $T_1\cdot T_2\cdot [\langle\sigma\rangle]=1$. From Theorem 3.11, then the classification of regular 1-cycle $T_2\cdot [\langle\sigma\rangle]$ means that there exists a tropical binomial $L^1(i,j)\le T_1$ which serves as the regular sequence of $T_2\cdot [\langle\sigma\rangle]$: $L^1(i,j)\cdot T_2\cdot[\langle\sigma\rangle]=1$. After swapping $T_2$ and $L^1(i,j)$ we obtain another tropical binomial $L^2(a,b)\le T_2$ which satisfies $L^2(a,b)\cdot L^1(i,j)\cdot[\langle\sigma\rangle]=1$. Therefore, $\sigma$ belongs to the gallery over pair $\{L^1(i,j),L^2(a,b)\}$, contradiction. The same reason allows us to obtain $T_1\cdot T_1\cdot [\langle\sigma\rangle]=0=T_2\cdot T_2\cdot [\langle\sigma\rangle]$ as required. However, these three equalities contradict with Lemma 4.7 b.

   Therefore, for any representative $F$ of $[F]$, $F$ is a covering of the galleries from $\mathrm{G}$. So the weight function of $F$ is the constant $1$. To show the finiteness of such $[F]$, it is enough to show that we only have finitely many supports $|F|$.

   Given any $\sigma\in F^{(2)}$. Lemma 4.11 gives
   $$T_1\cdot T_2\cdot[\langle\sigma\rangle]\le T_1\cdot T_2\cdot [F]\le1,$$
   $$T_1\cdot T_1\cdot[\langle\sigma\rangle]\le T_1\cdot T_1\cdot [F]\le1,$$
   $$T_2\cdot T_2\cdot [\langle\sigma\rangle]\le T_2\cdot T_2\cdot [F]\le 1.$$
   Theorem 4.4 and Lemma 4.7 directly reveal a finite list of all 2-planes $\langle\sigma\rangle$. In other words, we can only have finitely many ambient spaces $\langle\sigma\rangle$ for any $\sigma\in F^{(2)}$. Hence, we can only have finitely many supports $|F|$ since the coarsest representative $F$ of $[F]$ is a subfan of the union of these 2-planes (with weight $1$).

\end{proof}

\begin{theor}
    Assume that $M_1,M_2$ are two non-negative TR functions over $\mathbb{R}^n$ such that dim$(\mathrm{Newt}(M_1)+\mathrm{Newt}(M_2))=n$. Then there are finitely many 2-cycles $[F]$ such that $M_1\cdot M_1\cdot [F]\le 1$, $M_1\cdot M_2\cdot [F]\le 1$, and $M_2\cdot M_2\cdot [F]\le 1$.
\end{theor}

\begin{proof}
Let $[F]$ be a $2$-cycle over $\mathbb{R}^n$ that {satisfies} the assumptions above and let $M_{1},M_{2}$ be non-negative TR functions such that dim$(\mathrm{Newt}(M_1)+\mathrm{Newt}(M_2))=n$ with
\[
  M_{1}\!\cdot\!M_{2}\!\cdot\![F] \;=\;1.
\]
Then one can choose any non-negative TR functions $T_{1}\le M_{1}$ and $T_{2}\le M_{2}$ such that $$T_1:=\max(0,v_1,v_2,\dots,v_k),$$
$$T_2:=\max(0,w_1,w_2,\dots,w_{n-k}),$$
and $v_1,v_2,\dotsc,v_k,w_1,\dotsc,w_{n-k}$ are linearly independent over $(\mathbb{Z}^n)^\vee$. 
By monotonicity,
\[
  T_{1}\!\cdot\!T_{2}\!\cdot\![F] \;\le\; M_{1}\!\cdot\!M_{2}\!\cdot\![F] \;\le\;1.
\]
Similarly, we have
$$T_1\cdot T_1\cdot [F]\le M_1\cdot M_1\cdot [F]\le1,$$
$$T_2\cdot T_2\cdot [F]\le M_2\cdot M_2\cdot [F]\le1.$$
The three inequalities above are exactly the assumption stated in Theorem 4.12. Therefore, we can only have finitely many such $2$-cycles $[F]$ from Theorem 4.12 immediately.
\end{proof}

Finally, we obtain our desired finiteness theorem for the strongly regular $2$-cycles from Theorem 4.13 directly.

\begin{coro}[Finiteness theorem for strongly regular $2$-cycles]
    Assume that $M_1,M_2$ are two non-negative TR functions over $\mathbb{R}^n$ such that dim$(\mathrm{Newt}(M_1)+\mathrm{Newt}(M_2))=n$. Then there are finitely many $2$-cycles which are strongly regular over $M_1,M_2$.
\end{coro}

We conclude with an explicit example that lies outside both frameworks
considered in this paper.
The construction produces a $2$-cycle in $\mathbb{R}^{4}$ that is neither
strongly regular nor satisfies the Hodge index theorem.

Compare Section $5$ of \cite{BabaeeHuh19}, where the authors also give a
$2$-cycle in $\mathbb{R}^{4}$ violating the Hodge index theorem; their example
is considerably more sophisticated, being {irreducible}.
By contrast, the $2$-cycle presented below is reducible.

This raises a natural question:
can the methods leading to Theorems 4.12 and 4.13 be refined to yield an
{irreducible} $2$-cycle that fails the Hodge index theorem?

\begin{exa}
    We can have a $2$-cycle $[F]$ satisfying $M_1\cdot M_2\cdot [F]=0$ but $M_1\cdot M_1\cdot [F]=1=M_2\cdot M_2\cdot [F]$: Let $M_1:=\max(0,v_1,v_2)$ and $M_2:=\max(0,v_3,v_4)$ such that det$(v_1,v_2,v_3,v_4)=1$. Let $F$ be a 2-dimensional tropical fan whose support is $W_1\cup W_2$ where $W_1$ (resp. $W_2$) is the lineality space of $V(T_1)$ (resp. $V(T_2)$) and weight function is constant $1$. Then $M_2\cdot [F]=[V(M_2)\cap_{st} W_1]$ and $M_1\cdot [F]=[V(M_1)\cap_{st} W_2]$. Since both of them are $1$-D Bergman fans, so $M_1\cdot M_1\cdot [F]=1=M_2\cdot M_2\cdot [F]$. But $M_1\cdot [F]\subset W_2$, hence $M_2\cdot M_1\cdot [F]=0$.
\end{exa}

\printbibliography
\end{document}